\documentclass[final,1p,times]{article}
\usepackage[utf8]{inputenc}

\usepackage{latexsym,hyperref}
\usepackage{amssymb,amsbsy,amsmath,amsfonts,amssymb,amscd,amsthm,mathabx,stmaryrd}
\usepackage{epsfig,graphicx,subfigure}

\usepackage{pdfsync}
\usepackage{psfrag}
\usepackage{colortbl}

\usepackage[colorinlistoftodos,bordercolor=orange,backgroundcolor=orange!20,linecolor=orange,textsize=scriptsize]{todonotes}

\newcommand{\commentout}[1]{}
\newcommand{\R}{\mathbb{R}}
\newcommand{\N}{\mathbb{N}}
\newcommand{\Z}{\mathbb{Z}}

\newcommand {\e}  {\varepsilon}

\newcommand {\Chi} {{\bf \raise 2pt \hbox{$\chi$}} }
\newcommand {\cA} { {\mathcal A} }

\newcommand {\cU} { {\mathcal U} }

\newcommand {\f}   {\frac}
\newcommand {\p}   {\partial}

\newcommand{\dis}{\displaystyle}
\newcommand{\beq}{\begin{equation}}
\newcommand{\beqa} {\begin{array}{rl}}
\newcommand{\eeq}{\end{equation}}
\newcommand{\eeqa}{\end{array}}

\newcommand{\vertiii}[1]{{\left\vert\kern-0.25ex\left\vert\kern-0.25ex\left\vert #1 
    \right\vert\kern-0.25ex\right\vert\kern-0.25ex\right\vert}}
    
\newtheorem{theorem}{Theorem}
\newtheorem{lemma}{Lemma}
\newtheorem{remark}{Remark}
\newtheorem{proposition}[lemma]{Proposition}
\newtheorem{corollary}{Corollary}

\DeclareMathOperator*{\supess}{sup\,ess}

\definecolor{orange}{cmyk}{0,0.5,1,0.3}

\newcommand{\C}{\mathbb C}

\begin{document}

\title{Cyclic asymptotic behaviour of a population reproducing by fission into two equal parts}

\author{\'Etienne Bernard \thanks{Laboratoire de G\'eod\'esie, IGN-LAREG, B\^atiment Lamarck A et B, 35 rue H\'el\`ene Brion, 75013 Paris. Email: etienne.bernard@ign.fr}
 \and Marie Doumic \thanks{Sorbonne Universit\'es, Inria, UPMC Univ Paris 06, Mamba project-team, Laboratoire Jacques-Louis Lions. Email adress: marie.doumic@inria.fr} \thanks{Wolfgang Pauli Institute, c/o Faculty of Mathematics of the University of Vienna}
 \and Pierre Gabriel \thanks{Laboratoire de Math\'ematiques de Versailles, UVSQ, CNRS, Universit\'e Paris-Saclay,  45 Avenue des \'Etats-Unis, 78035 Versailles cedex, France. Email: pierre.gabriel@uvsq.fr}}

\maketitle

\begin{abstract}
We study the asymptotic behaviour of the following linear growth-fragmentation equation
$$ \f{\partial}{\partial t}  u(t,x) + \dfrac{\partial}{
\partial x} \big(x u(t,x)\big) + B(x) u(t,x) =4 B(2x)u(t,2x),$$
and prove that under fairly general assumptions on the division rate $B(x),$
its solution converges  towards an oscillatory function,
explicitely given  by the projection of the initial state on the space generated by the countable set of the dominant eigenvectors of the operator.
Despite the lack of hypocoercivity of the operator, the proof relies on a general relative entropy argument in a convenient weighted $L^2$ space, where well-posedness is obtained via semigroup analysis. We also propose a non-diffusive numerical scheme, able to capture the oscillations.
\end{abstract}


\medskip


\sloppy

\section*{Introduction}

Over the last decades, the mathematical study of the growth-fragmentation equation and its linear or nonlinear variants has led to a wide literature. 

Several facts  explain this lasting interest.
First, variants of this equation are used to model a wide range of applications, from the internet protocol suite to cell division or polymer growth; it is also obtained as a useful rescaling for the pure fragmentation equation (then the growth rate is linear). Second, despite the relative simplicity of such a one-dimensional equation, the study of its behaviour reveals complex and interesting interplays between growth and division, and a kind of dissipation even in the absence of diffusion. Finally, the underlying stochastic process has also - and for the same reasons - raised much interest, and only recently have the links between the probabilistic approach and the deterministic one begun to be investigated.

In its general linear form,  the equation may be written as follows
\begin{equation}\label{eq:generalform}\tag{GF}
\f{\partial}{\partial t}  u(t,x) + \dfrac{\partial}{
\partial x} \big(g(x) u\big) + B(x) u(t,x) =\int\limits_x^\infty k(y,x) B(y)u(t,y)dy, \end{equation} 
where $u(t,x)$ represents the concentration of individuals of size $x\geq 0$ at time~$t$, $g(x)\geq 0$ their growth rate, $B(x)\geq 0$ their division rate, and $k(y,x)\geq 0$ the quantity of individuals of size~$x$ created out of the division of individuals of size~$y.$ 

The long-time asymptotics of this equation has been studied and improved in many successive papers. Up to our knowledge, following the biophysical pioneering papers~\cite{BellAnderson, Bell, SinkoStreifer2}, the first mathematical study was carried out in~\cite{DiekmannHeijmansThieme1984}, where the equation was considered for the mitosis kernel / binary fission ($k(y,x)=2 \delta_{x=\f{y}{2}}$) in a compact set $x\in [\alpha,\beta].$ The authors proved the central  behaviour of the equation, already conjectured in~\cite{BellAnderson}: under balance and regularity assumptions on the coefficients, there exists a unique dominant eigenpair $(\cU(x)\geq 0, \lambda >0)$ such that $u(t,x)e^{-\lambda t} \to\cU(x)$ in a certain sense, with an exponential speed of convergence. In~\cite{DiekmannHeijmansThieme1984}, the proofs were based on semigroup methods and stated for the space of continuous functions provided with the supremum norm. Many studies followed:  some of them, most notably and recently~\cite{mischler:frag}, relaxing the previous assumptions in the context of semigroup theory~\cite{Arino,BanasiakArlotti,BanasiakLamb,BanasiakPichorRudnicki,nagel2,GreinerNagel}; others deriving explicit solutions~\cite{HallWake_1990,Zaidi_2015,Zaidi_2015-b} or introducing new methods - one of the most elegant and powerful being the \emph{General Relative Entropy}~\cite{MMP2}, leading to convergence results in norms weigthed by the adjoint eigenproblem. However, though in some cases the entropy method may lead to an explicit spectral gap in some integral norm~\cite{PR,LP,PPS3}, or when the coefficients are such that an entropy-entropy dissipation inequality exists~\cite{CCM,BCG}, in general it fails to provide a rate of convergence.

On the margins of this central behaviour, some papers investigated non-uniqueness~\cite{Banasiak2} or other kinds of asymptotics, happening for instance when the balance or mixing assumptions between growth and division fail to be satisfied: e.g. when the fragmentation dominates the growth~\cite{BanasiakLamb,Bertoin,BertoinWatson,Haas,DE}. A stronger ``memory'' of the initial behaviour may then be observed, contrary to the main case, where the only memory of the initial state  which remains asymptotically is a weighted average.

Among these results, the case when the growth rate is linear, \emph{i.e.} $g(x)=x,$ and the mother cell divides into two equal offspring, \emph{i.e.} $k(y,x)=2\delta_{x=\f{x}{2}},$ holds a special place, both for modelling reasons - it is the emblematic case of idealised bacterial division cycle, and also the rescaling adapted to the pure fragmentation equation - and as a limit case where the standard results fail to be true. The equation is then
\begin{equation}\label{eq:main}\left\{\begin{array}{l}
\dis\f{\partial}{\partial t}  u(t,x) + \dfrac{\partial}{
\partial x} \big(x u(t,x)\big) + B(x) u(t,x) =4 B(2x)u(t,2x),\qquad x>0,
\\ \\
u(0,x)=u^{\rm{in}}(x).\end{array}\right.
\end{equation} 
In 1967, 
G. I. Bell and E. C. Anderson already noted~\cite{BellAnderson}:  

{\it \small "If the
rate of cell growth is proportional to cell volume, then  (...) a daughter cell, having just half the volume of the parent cell, will grow at just half the rate of the parent cell. It follows that if one starts with
a group of cells of volume $V$, age $r$, at time $0$, then any daughter cell of this group, no matter when formed, will always have a volume equal to half the volume of an undivided cell in the group. There will then be no dispersion of cell volumes with time, and the population will consist at any time of a number of cell generations differing by just a factor of 2 in volume. For more general initial conditions, the population at late times will still reflect the initial state rather than simply growing exponentially in time. "}

After~\cite{BellAnderson}, the reason for this specific behaviour was stated in~\cite{DiekmannHeijmansThieme1984,Heijmans1985}: instead of a unique dominant eigenvalue, there exists a countable set  of dominant eigenvalues, namely $1+ \f{2i\pi}{\log 2}\Z.$ O. Diekmann, H. Heijmans and H. Thieme explain in~\cite{DiekmannHeijmansThieme1984}:

{\it \small 
The total population size still behaves like $[e^t]$ but convergence in shape does not take place. Instead the initial size distribution {\bf turns around and around} while numbers are multiplied. (...) The following Gedanken experiment illustrates the biological reason.
Consider two cells A and B with
equal size and assume that at some time instant $t_0$ cell $A$ splits into $a$ and $a$.
During the time interval $[t_0, t_1]$, $a$, $a$ and $B$ grow and at $t_1$ cell $B$ splits into $b$
and $b$. If $g(x) =cx$, the daughter cells $a$ and $b$ will have equal sizes just as their
mothers $A$ and $B$. ln other words, the relation "equal size" is hereditary and
extends over the generations. The growth model behaves like a multiplicating
machine which copies the size distribution.}

In~\cite{GreinerNagel}, G. Greiner and R. Nagel are the first to prove this long-time periodic behaviour.
They use the theory of positive semigroups combined with spectral analysis to get the convergence to a semigroup of rotations.
The method relies on some compactness arguments, which force the authors to set the equation on a compact subset of $(0,\infty)$ ($x\in[\alpha,\beta]$ with $\alpha>0$).

In the present paper, we extend the result to the equation set on the whole~$\R_+.$
Additionally, we determine explicitly the oscillatory limit by the means of a projection of the initial condition on the dominant eigenfunctions.
Our method relies on General Relative Entropy inequalities (Section~\ref{sec:GRE}), which unexpectedly may be adapted to this case and which are the key ingredient for an explicit convergence result (Theorem~\ref{theo:asympto:GRE} in Section~\ref{sec:L2}, which is the main result of our study).
We illustrate our results numerically in Section~\ref{sec:num}, proposing a non-diffusive scheme able to capture the oscillations.

\section{Eigenvalue problem and Entropy}\label{sec:GRE}

To study the long time asymptotics of Equation~\eqref{eq:main}, we elaborate on previously established results concerning the dominant positive eigenvector and general relative entropy inequalities. 

\subsection{Dominant eigenvalues and balance laws}

The eigenproblem and adjoint eigenproblem related to Equation~\eqref{eq:main} are
\begin{align}\label{eq:vep}
&\lambda \cU(x) + \big(x\cU(x)\big)' + B(x)\cU (x)=4B(2x)\cU(2x),\\
\label{eq:vepadj}
&\lambda \phi(x) - x \phi'(x) + B(x)\phi (x)=2B(x)\phi\Big(\f{x}{2}\Big).
\end{align}
Perron eigenproblem consists in finding positive solutions $\cU$ to~\eqref{eq:vep}, which in general give the asymptotic behaviour of time-dependent solutions, which align along $e^{\lambda t} \cU (x)$. Recognizing here a specific case of the eigenproblem studied in~\cite{DG}, we work under the following assumptions:
\begin{equation}\label{hyp:B}\left\{
\begin{array}{l}
B: (0,\infty) \to (0,\infty) \text{ is locally integrable}, \\  \\
\text{$\exists z_0,\gamma_0,K_0>0,\ \forall x<z_0,\quad B(x)\leq K_0 x^{\gamma_0}$},\\ \\
\text{$\exists z_1,\gamma_1,\gamma_2,K_1,K_2>0,\ \forall x>z_1,\quad K_1x^{\gamma_1}\leq B(x)\leq K_2 x^{\gamma_2}$}.\end{array}\right.
\end{equation}
We then have the following result, which is a particular case of~\cite[Theorem~1]{DG}.

\begin{theorem}\label{th:Perron}
Under Assumption~\eqref{hyp:B}, there exists a unique positive eigenvector $\cU \in L^1(\R_+)$ to~\eqref{eq:vep} normalised by $\int_0^\infty x\, \cU (x)dx=1$. It is related to the eigenvalue $\lambda=1$ and to the adjoint eigenvector $\phi(x)=x$ solution to~\eqref{eq:vepadj}. 

Moreover,   $x^\alpha \cU \in L^\infty(\R_+)$ for all $\alpha\in \R$, and $ \cU \in W^{1,1} (\R_+).$ 
\end{theorem}
As already noticed in~\cite{DiekmannHeijmansThieme1984}, though $1$ is the unique eigenvalue related to a positive eigenvector, here it is not the unique dominant eigenvalue: we have a set of eigentriplets $(\lambda_k,\cU_k,\phi_k)$ with $k\in \Z$ defined by
\begin{equation}\label{def:vepk}
\lambda_k=1+\f{2ik\pi}{\log 2},\qquad \cU_k(x)=x^{-\f{2ik \pi}{\log 2}} \cU(x),\qquad \phi_k (x)= x^{1+\f{2ik\pi}{\log 2}}.
\end{equation}
This is  the first difference with the most studied case, where the Perron eigenvalue happens to be the unique dominant one: here all these eigenvalues have a real part equal to $1$, so that they all belong to the peripheral spectrum. The natural questions which emerge are to know whether this set of dominant eigenvectors is attractive, as it is the case when it is formed by a unique function; and if so, where the proofs are different.

First we notice an important property:
the family $\big((\cU_k)_{k\in\Z},(\phi_k)_{k\in\Z}\big)$ is biorthogonal for the bracket
\[\langle f,\varphi\rangle:=\int_0^\infty f(x)\varphi(x)\,dx,\]
which means that
\beq\label{eq:biorthogonal}
\forall(k,l)\in\Z^2,\qquad\langle\cU_k,\phi_l\rangle=\delta_{kl}.
\eeq
This  is a direct consequence of the normalization of the Perron eigenvectors which writes $\langle\cU,\phi\rangle=1$ and the fact that $\lambda_k\neq\lambda_l$ for $k\neq l.$

\medskip

Even though we are interested in real-valued solutions to Equation~\eqref{eq:main}, due to
the fact that the dominant eigenelements have nonzero imaginary part, we have to work in spaces of complex-valued functions.
Of course real-valued solutions are readily obtained from complex-valued solutions by taking the real or imaginary part.
From now on when defining functional spaces we always consider measurable functions from $\R_+$ to $\C.$

\medskip

The biorthogonal property~\eqref{eq:biorthogonal} can be extended into balance laws for general solutions to Equation~\eqref{eq:main}.
For $u^{\rm{in}} \in L^1(\phi(x)dx)$ and $u \in {\mathcal C} (\R_+,L^1(\phi(x)dx))$ solution to~\eqref{eq:main} we have the conservation laws
\beq\label{eq:balance}
\forall k\in\Z,\ \forall t\geq0,\qquad\langle u(t,\cdot),\phi_k\rangle\,e^{-\lambda_kt}=\langle u^{\rm{in}},\phi_k\rangle.
\eeq

\subsection{General Relative Entropy inequalities}

Additionally to the conservation laws above, we have a set of entropy inequalities.
In this section, we remain at a formal level.
Rigorous justification of the stated results will appear once the existence and uniqueness results are established.
\begin{lemma}[General Relative Entropy Inequality] \label{lem:GRE}
Let $B$ satisfy Assumption~\eqref{hyp:B}, $\cU$ be the Perron eigenvector defined in Theorem~\ref{th:Perron} and $u(t,x)$ be a solution of Equation~\eqref{eq:main}. Let $H:\C\to \R_+$ be a positive, differentiable and convex function. Provided the quantities exist, we have
$$\f{d}{dt} \int\limits_0^\infty x\,\cU (x) H\Big(\f{u(t,x)}{\cU(x)e^t}\Big)   dx
=-D^H[u(t)e^{-t}]\leq 0,$$
with $D^H$ defined by
\begin{align*}
D^H[u]:=\int\limits_0^\infty xB(x)\, \cU(x)&\biggl[
H\Big(\f{u(\f{x}{2})}{\cU(\f{x}{2})}\Big) - H\Big(\f{u(x)}{\cU(x)}\Big) \nonumber\\
& \hskip8mm - \nabla H\Big(\f{u(x)}{\cU(x)}\Big)\cdot\Big(\f{u(\f{x}{2})}{\cU(\f{x}{2})}-\f{u(x)}{\cU(x)}\Big)
\biggr]dx,
\end{align*}
where $\nabla H$ is the gradient of $H$ obtained by identifying $\C$ with $\R^2$
and $\,\cdot$ stands for the canonical inner product in $\R^2.$
Moreover, for $H$ strictly convex, $u: \R_+ \to \C$ satisfies $D^H[u]=0$ \emph{iff} it is such that 
\[\f{u(x)}{\cU(x)}=\f{u(2x)}{\cU(2x)},\qquad \text{a.e. }\;x >0.\]
In particular, for all $k\in\Z,$ $D^H[\cU_k]=0.$
\end{lemma}

The proof is immediate and now standard, carried out by calculation term by term and use of the equations~\eqref{eq:main}, \eqref{eq:vep} and~\eqref{eq:vepadj}, see for instance~\cite[p.92]{BP}.

In the cases where the Perron eigenvector is a unique dominant eigenvector, the entropy inequality  is a key step to obtain the convergence of $u(t,x)e^{-t}$ towards $\langle u^{\rm{in}},\phi\rangle\,\cU (x)$.  The idea is to prove that $u(t,x)e^{-t}$ tends to a limit $u_\infty$ such that $D^H[u_\infty]=0,$ which in general implies that $u_\infty$ is proportional to $\cU;$ the conservation law then giving the proportionality constant.

Here however, since any function $v(x)=f(\log x)\cU(x)$ with $f$ $\log 2$-periodic satisfies $D^H[v]=0,$ the usual convergence result does not hold.
This is due to the lack of hypocoercivity in our case.
It is known from~\cite{CCM,BCG,GS14} that the general form~\eqref{eq:generalform} of the growth-fragmentation equation is coercive for particular choices of the coefficients, in the sense that the differential inequality
$\frac{d}{dt}\|u(t,\cdot)\|\leq -\nu \|u(t,\cdot)\|$ holds for some positive constant $\nu$ and a well-chosen norm $\|\cdot\|,$ when $u^{\rm{in}}$ is such that $\langle u^{\rm{in}},\phi\rangle=0.$
As already noticed in~\cite{LP} such an inequality cannot be valid for an entropic norm in the case of equal mitosis.
Indeed if for some time $t\geq0$ (for instance $t=0$) the solution satisfies $u(t,x)/\cU(x)=u(t,2x)/\cU(2x),$ then the time derivative of the norm vanishes.
However in this case the equation can be hypocoercive in the sense (see~\cite{Villani}) that $\|u(t,\cdot)\|\leq Ce^{-\nu t}\|u^{\rm{in}}\|$ holds for some positive constants $C,\nu$ and any initial distribution satisfying $\langle u^{\rm{in}},\phi\rangle=0.$ This result is proved in~\cite{mischler:frag,BernardGabriel_2} for a class of weighted $L^1$ norms in the case of a constant growth rate $g.$
Roughly speaking this situation of a non-coercive but hypocoercive equation appears when the dissipation of entropy can vanish for a nontrivial set of functions, but this set is unstable for the dynamics of the equation.
In our case the equation is not hypocoercive because the set of functions with null entropy dissipation is invariant under the flow, as expressed by the following lemma.
\begin{lemma}\label{lm:nonhypocoercivity}
Consider a strictly convex function $H$ and let $u(t,x)$ be the solution to Equation~\eqref{eq:main} with initial condition $u(0,x)=u^{\rm{in}}(x).$
We have the invariance result
\[D^H[u^{\rm{in}}]=0\qquad\implies\qquad D^H[u(t,\cdot)]=0,\quad\forall t\geq0.\]
\end{lemma}

As Lemma~\ref{lem:GRE}, Lemma~\ref{lm:nonhypocoercivity} is valid in a space where the existence and uniqueness of a solution is proved, as for instance in the space $L^2(\R_+,x/\cU(x)dx),$ see Section~\ref{sec:L2}.
\begin{proof}
Let $u^{\rm{in}}$ such that $D^H[u^{\rm{in}}]=0$ and denote $u(t,\cdot)$ the solution to Equation~\eqref{eq:main}.

We have already seen in Lemma~\ref{lem:GRE} that for any $u:\R_+\to\C,$ we have
\[D^H[u]=0\qquad\Longleftrightarrow\qquad \f{u(x)}{\cU(x)}=\f{u(2x)}{\cU(2x)},\quad \text{a.e. }\ x >0,\]
so that by assumption $\f{u^{\rm{in}}(x)}{\cU(x)}=\f{u^{\rm{in}}(2x)}{\cU(2x)}$ for almost every $x>0.$ 

To prove
 Lemma~\ref{lm:nonhypocoercivity} we thus want to prove that $\f{u(t,x)}{\cU(x)}=\f{u(t,2x)}{\cU(2x)}$ for almost every $x>0,$ $t>0.$ To do so, we notice that if we have a solution $\tilde u$ of Equation~\eqref{eq:main} which satisfies this property, then the ration $v(t,x)=\tilde u(t,x)/\cU (x) e^{-t}$ is solution of the following simple transport equation
\[\p_t v(t,x)+x\p_xv(t,x)=0,\]
so that $v(t,x)=v(0,xe^{-t}).$
We are led to define a function $u_1$ by
\[
u_1(t,x):=u^{\rm{in}}(xe^{-t})\frac{\cU(x)\,e^t}{\cU(xe^{-t})}.
\]
We easily check that $u_1(t,x)/\cU(x)=u_1(t,2x)/\cU(2x)$ for all $t$ and almost all $x,$ and that $u_1$  is solution to Equation~\eqref{eq:main}. We conclude  by uniqueness that we have $u\equiv u_1$ and so $D^H[u(t,\cdot)]=0$ for all $t\geq 0.$

\end{proof}

\

For $H(z)=|z|^p$ the entropy corresponds to the $p$-power of the norm in
\[E_p:=L^p(\R_+,\phi(x)\cU^{1-p}(x)\,dx).\]
Define also the space
\[E_\infty:=\bigl\{u:\R_+\to\C\ \text{measurable},\,\exists C>0,\,|u|\leq C\cU\ a.e.\bigr\},\]
which is the analogous of $E_p$ for $p=\infty,$
endowed with the norm
\[\|u\|_{E_\infty}:=\supess_{x>0}\frac{|u(x)|}{\cU(x)}.\]
These spaces have the property to be invariant under the dynamics of Equation~\eqref{eq:main} and to constitute a tower of continuous inclusions, as it is made more precise in the following two lemmas.

\begin{lemma}\label{lm:E_p-stability}
Let $p\in[1,\infty]$ and let $u(t,x)$ be the solution to Equation~\eqref{eq:main} with initial data $u^{\rm{in}}\in E_p.$
Then $u(t,\cdot)\in E_p$ for all $t\geq0$ and
\[\|u(t,\cdot)e^{-t}\|_{E_p}\leq\|u^{\rm{in}}\|_{E_p}.\]
\end{lemma}

\begin{proof}
For $p<\infty,$ this is a direct consequence of Lemma~\ref{lem:GRE} by considering the convex function $H(z)=|z|^p.$
Similarly for $p=\infty$ we get the result by applying Lemma~\ref{lem:GRE} with the convex function
\[H(z)=\left\{\begin{array}{ll}
|z|-C&\text{if}\ |z|\geq C
\vspace{2mm}\\
0&\text{if}\ |z|\leq C
\end{array}\right.\]
with $C=\|u^{\rm{in}}\|_{E_\infty}.$
\end{proof}

\begin{lemma}
Let $1\leq p\leq q\leq\infty$ and $u\in E_q.$
Then $u\in E_p$ and
\[\|u\|_{E_p}\leq\|u\|_{E_q}. \]
\end{lemma}

\begin{proof}
It is clear if $q=+\infty.$
For $q<+\infty,$ since $\phi(x)\,\cU(x)dx$ is a probability measure the Jensen's inequality ensures that
\[\|u\|_{E_p}^q=\bigg(\int \Big|\frac{u}{\cU}\Big|^p\phi\,\cU\bigg)^{q/p}\leq\int \Big|\frac{u}{\cU}\Big|^q\phi\,\cU=\|u\|_{E_q}^q.\]
\end{proof}

\section{Convergence in the quadratic norm}\label{sec:L2}
Equipped with the  General Relative Entropy inequalities, we  now combine them with Hilbert space techniques to prove the convergence to periodic solutions.
The Hilbert space formalism provides an interpretation of the periodic limit in terms of Fourier decomposition, and allows us to give the main ingredients of the proof while avoiding too many technicalities.
We first introduce the Hilbert space (Section~\ref{subsec:Hilbert}), in which we prove the well-posedness of Equation~\eqref{eq:main} (Section~\ref{subsec:well-posedness}). We state our main result in Theorem~\ref{theo:asympto:GRE}.

\subsection{The Hilbert space}\label{subsec:Hilbert}

As we will see below, working in a Hilbert setting is very convenient for our study.
Drawing inspiration from the General Relative Entropy with the convex quadratic function $H(z)=|z|^2,$
we work in the Hilbert space
\[E_2=L^2(\R_+,x/\cU(x)\,dx)\]
endowed with the inner product
\[(f,g):=\int_0^\infty f(x)\overline g(x)\frac{x}{\cU(x)}\,dx.\]
We denote by $\|\cdot\|$ the corresponding norm defined by
\[\|f\|^2=(f,f).\]
In this space, the normalization we have chosen for $\cU$ means
\[\|\cU\|=\|\cU_k\|=1\]
and the biorthogonality property~\eqref{eq:biorthogonal} reads
\[(\cU_k,\cU_l)=\langle\cU_k,\phi_l\rangle=\delta_{k,l},\]
meaning that $(\cU_k)_{k\in\Z}$ is an orthonormal family in $E_2.$
As a consequence the family $(\cU_k)_{k\in\Z}$ is a Hilbert basis of the Hilbert space
\[X:=\overline{\mbox{span}}(\cU_{k})_{k\in \Z}\]
and the orthogonal projection on this closed subspace of $E_2$ is given by
\[Pu:=\sum_{k=-\infty}^{+\infty}(u,\cU_k)\cU_k,\qquad\forall u\in E_2.\]
Additionally, we have the Bessel's inequality
\[
\|Pu\|^2=\sum_{k=-\infty}^{+\infty}|(u,\cU_k)|^2\leq\|u\|^2.
\]
As it is stated in the following lemma, there is a crucial link between $X$ and the quadratic dissipation of entropy ({\it i.e.} $D^H[u]$ for $H(z)=|z|^2$), which can be written in a simpler way as
\beq\label{def:D2}
D^2[u]=\int_0^\infty xB(x)\,\cU(x)\bigg|\frac{u(x)}{\cU(x)}-\frac{u(x/2)}{\cU(x/2)}\bigg|^2\,dx.
\eeq

\begin{lemma}\label{lm:X}
We have
\[X=\{u\in E_2,\ D^2[u]=0\}.\]
\end{lemma}

\begin{proof}
Since $|z|^2$ is strictly convex, we have already seen in Lemma~\ref{lem:GRE} (and it is even clearer in the case of $D^2$) that
\[\{u\in E_2,\ D^2[u]=0\}=\{u\in E_2,\ u(x)/\cU(x)=u(2x)/\cU(2x),\ a.e.\ x>0\}\supset X.\]
Also we clearly have
\begin{align*}
\{u\in E_2,&\ u(x)/\cU(x)=u(2x)/\cU(2x),\ a.e.\ x>0\}=\\
&\{u\in E_2,\,\exists f:\R\to\C\ \log2\text{-periodic},\,u(x)=f(\log x)\cU(x),\ a.e.\ x>0 \}.
\end{align*}
If $u\in E_2$ is of the form $u(x)=f(\log x)\cU(x)$ with $f:\R\to\C$ $\log2$-periodic then necessarily $f\in L^2([0,\log2])$
and the Fourier theory ensures (Fourier-Riesz-Fischer theorem) that
\[f(y)=\sum_{k=-\infty}^{+\infty}\hat f(k)e^{\frac{2ik\pi y}{\log2}},\]
where
\[\hat f(k)=\frac{1}{\log2}\int_{0}^{\log2}f(y)e^{-\frac{2ik\pi y}{\log2}}dy\in \ell^2(\Z).\]
So we have in $L^2_{loc}(0,\infty)$
\[u(x)=\cU(x)\sum_{k=-\infty}^{+\infty}\hat f(k)x^{\frac{2ik\pi}{\log2}}=\sum_{k=-\infty}^{+\infty}\hat f(-k)\cU_k(x)\in X.\]
We also deduce that $\hat f(k)=(u,\cU_{-k}).$
\end{proof}

\subsection{Well-posedness of the Cauchy problem}
\label{subsec:well-posedness}
Since the Perron eigenvalue $\lambda=1$ is strictly positive, it is convenient to consider a \emph{rescaled} version of our problem
\begin{equation}\label{eq:rescaled}\left\{\begin{array}{l}
\dis\f{\partial}{\partial t}  v(t,x) + \dfrac{\partial}{
\partial x} \big(x v(t,x)\big) + v(t,x) + B(x) v(t,x) =4 B(2x)v(t,2x),\quad x>0,
\\ \\
v(0,x)=u^{\rm{in}}(x).\end{array}\right.
\end{equation} 

The solutions to Equation~\eqref{eq:main} are related to the solutions to~\eqref{eq:rescaled} by the simple relation
\[u(t,x)=e^tv(t,x).\]
It is proved in~\cite{EscoMischler3} (see also~\cite{BernardGabriel_2}) that the problem~\eqref{eq:rescaled} is well-posed in $E_1$
and admits an associated $C_0$-semigroup $(T_t)_{t\geq0}$ which is positive,
meaning that for any $u^{\rm{in}}\in E_1$ there exists a unique (mild) solution $v\in C(\R_+,E_1)$ to~\eqref{eq:rescaled} which is given by $v(t)=T_tu^{\rm{in}},$
and $v(t)\geq0,$ $t\geq0,$ for $ u^{\rm{in}}\geq0.$
From Lemma~\ref{lm:E_p-stability} we have that all subspaces $E_p$ with $p\in[1,\infty]$ are invariant under $(T_t)_{t\geq0}.$
Additionally, the restriction of $T_t$ to any $E_p$ is a contraction, {\it i.e.}
\beq\label{eq:contraction}
\forall u\in E_p,\ \forall t\geq0,\qquad \|T_tu\|_{E_p}\leq\|u\|_{E_p}.
\eeq
To get the well-posedness of~\eqref{eq:rescaled} in $E_2,$ it only remains to check the strong continuity of $(T_t)_{t\geq0}$ in $E_2.$

\begin{lemma}
The semigroup $(T_t)_{t\geq0}$ restricted to $E_2$ is strongly continuous.
\end{lemma}

\begin{proof}

We use the subspace $E_\infty\subset E_2$ and the contraction property~\eqref{eq:contraction} to write
for any $u\in E_\infty$
\begin{align*}
\|T_tu-u\|_{E_2}^2=\int_0^\infty &|T_tu-u|^2(x)\frac{x}{\cU(x)}\,dx\\
&\leq2\|u\|_{E_\infty}\int_0^\infty|T_tu-u|(x)x\,dx=2\|u\|_{E_\infty}\|T_tu-u\|_{E_1}.
\end{align*}
The strong continuity of $(T_t)_{t\geq0}$ in $E_1$ ensures that $\|T_tu-u\|_{E_1}\to0$ and so $\|T_tu-u\|_{E_2}\to0$ when $t\to0.$
By density of $E_\infty\subset E_2$ we get the strong continuity of $(T_t)_{t\geq0}$ in $E_2.$

\end{proof}

We denote by $\cA$ the generator of the semigroup $(T_t)_{t\geq0}$ in $E_2.$
For any $u$ in the domain $D(\cA)$ we have in the distributional sense
\[\cA u(x)=-(xu(x))'-u(x)-B(x)u(x)+4B(2x)u(2x).\]
The eigenpairs $(\lambda_k,\cU_k)$ are defined by $\cA\,\cU_k=\text{$(\lambda_k-1)$}\,\cU_k$ and we easily prove the following properties.
\begin{proposition}\label{prop:Ttproperties}
For all $t\geq0$ we have
\begin{enumerate}
\item $\forall k\in\Z,\quad T_t\,\cU_k=e^{\frac{2ik\pi t}{\log2}}\,\cU_k,$
\item $\forall u\in E_2,\ \forall k\in\Z,\quad(T_tu,\cU_k)=(u,\cU_k)e^{\frac{2ik\pi t}{\log2}},$
\item $\forall u\in E_2,\quad PT_tu=T_tPu=\sum_{k\in\Z}(u,\cU_k)e^{\frac{2ik\pi t}{\log2}}\cU_k,$
\item $T_t X\subset X$ and for all $u\in X,\quad T_tu=\sum_{k\in\Z}(u,\cU_k)e^{\frac{2ik\pi t}{\log2}}\cU_k,$
\item $X\subset D(\cA)$ and for all $u\in X,\quad \cA u=\sum_{k\in\Z}(u,\cU_k)\frac{2ik\pi}{\log2}\cU_k.$
\end{enumerate}
\end{proposition}
The second property is nothing but a rewriting of the conservation laws~\eqref{eq:balance}.
The fourth point makes more precise and proves more rigorously the invariance property in Lemma~\ref{lm:nonhypocoercivity}.

\subsection{Convergence}
We are now ready to state the asymptotic behaviour of solutions to Problem~\ref{eq:main}.
\begin{theorem}\label{theo:asympto:GRE}
Assume that $B$ satisfies Hypothesis~\eqref{hyp:B} and define $\cU_k$ by~\eqref{def:vepk}.
Then for any $ u^{\rm{in}} \in E_2,$ the unique solution $u(t,x) \in C\big(\R_+,E_2\big)$ to Equation~\eqref{eq:main} satisfies
\[\int\limits_0^\infty\, \bigg|u(t,x)e^{-t} - \sum_{k=-\infty}^{+\infty}( u^{\rm{in}},\cU_k) e^{\frac{2ik\pi}{\log2}t}\,\cU_k(x) \bigg|^2\frac{x\,dx}{\cU(x)} \xrightarrow[t\to+\infty]{}0.\]
\end{theorem}

\

\begin{remark}
This convergence result can also be formulated in terms of semigroups.
Set
\[R_tu:=T_tPu=PT_tu=\sum_{k=-\infty}^\infty(u,\cU_k)e^{\frac{2ik\pi}{\log2}t}\,\cU_k.\]
This defines a semigroup,
\[R_{t+s}u=T_{t+s}Pu=T_{t+s}P^2u=T_tT_sP^2u=T_tPT_sPu=R_tR_su,\]
which is $\log2$-periodic.
The result of Theorem~\ref{theo:asympto:GRE} is equivalent to the strong convergence of $(T_t)_{t\geq0}$ to $(R_t)_{t\geq0}$, {\it i.e.}
\[\forall u\in E_2,\qquad\|T_tu-R_tu\|\xrightarrow[t\to+\infty]{}0.\]
It is also equivalent to the strong stability of $(T_t)_{t\geq0}$ in $X^\bot={\rm Ker}\,P$
\[\forall u\in X^\bot,\qquad\|T_tu\|\xrightarrow[t\to+\infty]{}0.\]
\end{remark}
\medskip

\begin{remark}
We may use the Poisson summation formula to reinterpret the limit function in terms of only $ u^{\rm{in}}(x):$  we recall that this formula states that, under proper assumptions on $f$ and its Fourier transform ${\mathcal F} f(\xi)=\int\limits_{-\infty}^{+\infty}f(y)e^{-iy\xi}dy,$ we have
$$\sum\limits_{\ell=-\infty}^{\infty} f(y+\ell a)=\sum\limits_{k=-\infty}^\infty {\mathcal F} f (\f{2\pi k}{a})e^{\f{2ik\pi y}{a}}.$$
Taking $a=\log 2,$ $f(y)= u^{\rm{in}}(e^{-y})e^{-2y},$ we apply it to the limit function taken in $y=t-\log x$ 
$$\sum\limits_{k=-\infty}^\infty  ( u^{\rm{in}},\cU_k)\cU_k (x) e^{\f{2ik\pi t}{\log 2}}= \cU(x) \sum\limits_{\ell=-\infty}^\infty 2^{-2\ell} x^2 e^{-2t}  u^{\rm{in}}(2^{-\ell} xe^{-t}).$$
This formula is reminiscent of a similar one found in~\cite{DE}, Theorem 1.3. (b), for the limit case $B$ constant. 
\end{remark}

\begin{proof}[Proof of Theorem~\ref{theo:asympto:GRE}]
We follow here the classical proof of convergence, pioneered in \cite{MMP1,MMP2}. Though the limit is now an oscillating function, this strategy may be adapted here, as shown below. 

Define
\[h(t,x):=u(t,x)e^{-t}-\sum_{k=-\infty}^{+\infty}( u^{\rm{in}},\cU_k) e^{\frac{2ik\pi}{\log2}t}\,\cU_k(x)=(I-P)T_t u^{\rm{in}}\]
which is solution to Equation~\eqref{subsec:well-posedness}.
Lemma~\ref{lm:E_p-stability} with $p=2$ ensures that
\[\f{d}{dt} \|h(t,\cdot)\| \leq 0,\]
so that it decreases through time. Since it is a nonnegative quantity, it means that it tends toward a limit $L\geq 0$ and it remains to show that $L=0.$
Let us adapt to our case the proof in  B. Perthame's book \cite[p.98]{BP}.
Because of the contraction property, it is sufficient to do so for $ u^{\rm{in}}\in D(\cA)$ which is a dense subspace of $E_2.$
Recall that for $ u^{\rm{in}}\in D(\cA)$ the solution to Equation~\eqref{eq:main} can be understood as a classical solution, $u(t,\cdot)$ belonging to $D(\cA)$ for all time.
The last property in Proposition~\ref{prop:Ttproperties} ensures that $X\subset D(\cA),$ so $h(0,\cdot)\in D(\cA).$
Define $q(t,x)=\partial_th(t,x)$ which is clearly a mild solution to Equation~\eqref{eq:main} with initial datum
\[q(t=0,x)=\cA h(t=0,x).\]
By contraction we get
\[\|q(t,\cdot)\|\leq\|\cA h(0,\cdot)\|.\]
Introduce the sequence of functions $h_n(t,\cdot)=h(t+n,\cdot).$
Since $h$ and $\p_th$ are uniformly bounded in the Hilbert space $E_2,$ the Ascoli and Banach-Alaoglu theorems ensure that $(h_n)_{n\in\N}$ is relatively compact in $C([0,T],E_2^{\rm w})$ where $E_2^{\rm w}$ is $E_2$ endowed with the weak topology.
After extracting a subsequence, still denoted $h_n,$ we have $h_n\to g$ in $C([0,T],E_2^{\rm w}).$
Additionally since $\int_0^\infty D^2[h(t,\cdot)]\,dt<+\infty,$ we have
\[\int_0^TD^2[h_n(t,\cdot)]\,dt=\int_{n}^{T+n}D^2[h(t,\cdot)]\,dt\to0.\]
and it ensures, using the definition~\eqref{def:D2} of $D^2,$ that $\frac{h_n(t,x)}{\cU(x)}-\frac{h_n(t,2x)}{\cU(2x)} \to0$ in the distributional sense.
We deduce from the convergence $h_n\to g$ that $\frac{g(t,x)}{\cU(x)}-\frac{g(t,2x)}{\cU(2x)}=0,$ and so $D^2[g(t,\cdot)]=0$ for all $t\geq0.$
By Lemma~\ref{lm:X} this means that $g(t,\cdot)\in X$ for all $t\geq0.$
But for all $n\in\N$ and all $t\geq0$ we have $h_n(t,\cdot)\in X^\bot={\rm Ker}P$ by construction of $h,$
and since $X^\bot$ is a linear subspace, the weak limit $g$ of $h_n$ also satisfies $g(t,\cdot)\in X^\bot$ for all $t\geq0.$
Finally $g(t,\cdot)\in X\cap X^\bot=\{0\}$ for all $t\geq0,$ so $g\equiv0$ and the proof is complete.
\end{proof}

The result in Theorem~\ref{theo:asympto:GRE} is in contrast to the property of asynchronous exponential growth which states that the solutions behave like $u(t,x)\sim\langle  u^{\rm{in}},\phi\rangle\,\cU(x) e^t$ when $t\to+\infty.$
This property is satisfied for a large class of growth-fragmentation equations~\cite{MMP2}, but the lack of hypocoercivity in our case prevents it to hold.
However we can deduce from Theorem~\ref{theo:asympto:GRE} a ``mean asynchronous exponential growth'' property, in line with  probabilistic results, e.g.~\cite{DHKR}.

\begin{corollary}\label{cor:mean_ergodicity}
Under Assumption~\eqref{hyp:B}, the semigroup $(T_t)_{t\geq0}$ generated by $(\cA,D(\cA))$ is \emph{mean ergodic}, {\it i.e.}
\[\forall u\in E_2,\qquad \frac1t\int_0^tT_su\,ds\xrightarrow[t\to+\infty]{}P_0u=(u,\cU)\,\cU=\langle u,\phi\rangle\,\cU.\]
\end{corollary}

\begin{proof}
Because of Theorem~\ref{theo:asympto:GRE}, it suffices to prove that
\[\frac 1t\int_0^tR_su\,ds=P\bigg(\frac 1t\int_0^tT_su\,ds\bigg)\xrightarrow[t\to+\infty]{}P_0u.\]
Denoting $m_t=\frac 1t\int_0^tT_su\,ds$ the Ces\`aro means of $(T_tu)_{t\geq0}$ we have
\[Pm_t=\sum_{k=-\infty}^{+\infty}(m_t,\cU_k)\cU_k.\]
By the conservation laws~\eqref{eq:balance} we have for $k\neq0$
\[(m_t,\cU_k)=\frac1t\int_0^t(T_su,\cU_k)\,ds=(u,\cU_k)\frac1t\int_0^te^{\frac{2ik\pi}{\log2}s}ds=(u,\cU_k)\frac{\log2}{2ik\pi}\frac{e^{\frac{2ik\pi}{\log2}t}-1}{t}\]
and $(m_t,\cU_0)=(u,\cU_0).$
This gives
\[Pm_t=P_0u+\frac1t\sum_{k\neq0}(u,\cU_k)\frac{\log2}{2ik\pi}\,\cU_k\big(e^{\frac{2ik\pi}{\log2}t}-1\big).
\]
Since
\begin{align*}
\bigg\|\sum_{k\neq0}(u,\cU_k)\frac{\log2}{2ik\pi}\,\cU_k\big(e^{\frac{2ik\pi}{\log2}t}-1\big)\bigg\|^2&=
\sum_{k\neq0}\Big|(u,\cU_k)\frac{\log2}{2ik\pi}\big(e^{\frac{2ik\pi}{\log2}t}-1\big)\Big|^2\\
&\leq\Big(\frac{\log2}{\pi}\Big)^2\sum_{k\in\Z}|(u,\cU_k)|^2=\Big(\frac{\log2}{\pi}\Big)^2\|Pu\|^2
\end{align*}
we conclude that
\[\|Pm_t-P_0u\|\leq\frac1t\frac{\log2}{\pi}\|Pu\|\xrightarrow[t\to+\infty]{}0.\]
\end{proof}

\section{Numerical solution}\label{sec:num}

\subsection{A first-order non diffusive numerical scheme}
Another way to understand the origin of the oscillatory behaviour is to consider the underlying Piecewise Deterministic Markov Process (PDMP), see e.g.~\cite{BertoinWatson,Cloez,DHKR}.
If we follow a given cell of size $x$ at time $0,$ it is of size $2^{-n} xe^t$ at time $t$ if it has divided $n$ times before $t$; hence,
any of its descendants  has to remain exactly in the countable set $\{2^{-\ell} xe^t,$ $\ell\in \N\}$ at any time.
We can say that we need a ``non-diffusive'' numerical scheme: if the transport rate is not exactly linear but {\it approximately} linear, or if the splitting into two cells does not give rise to two exactly equally-sized but to {\it approximately} two equally-sized daughters,  then the numerical scheme computes the solution of an {\it approximate} equation, which is proved, after renormalization, to converge exponentially fast toward a steady behaviour. Looking at the descendants, it means that instead of remaining in the countable set $\{2^{-\ell} xe^t,$ $\ell\in \N\}$, they will disperse around, and progressively fill in the space $(0,x e^t).$
 This exponential convergence toward a steady state will give rise only to some damped oscillations.

\medskip

The numerical scheme thus needs to satisfy the two following {conditions}:
\begin{enumerate}
\item the { discretization of the} transport equation $\f{\p}{\p t} u +\f{\p}{\p x} (xu)$ { must be non diffusive. If we use a standard upwind scheme, we would thus like to have a } Courant-Friedrichs-L\'evy (CFL) condition equal to $1.$ { This means that any  point of the grid at time $t$ is transported by the transport equation $\f{\p}{\p t} u +\f{\p}{\p x} (xu)$  to another point of the grid at time $t+\Delta t$.}
\item \label{cond2} The { discretization of the } fragmentation term $4 B(2x)u(t,2x)-B(x)u(t,x)$ { must ensure that if $x$ is a point of the grid, then so is $x/2$ and $2x$ - at least inside the computational domain $[x_{min},x_{max}]$ - so that there is no approximation when applying the fragmentation operator.}
\end{enumerate}
The condition~\ref{cond2} leads us to define the following geometric grid, for given $n,N\in\N^*$:  \begin{equation}\label{eq:grid}
\delta x:=2^{\f{1}{n}}-1,\qquad x_k:=(1+\delta x)^{k-N},\qquad 0\leq k \leq 2N.
\end{equation}
Then, for any $k\in \N,$ $0\leq k\leq 2N,$ $2x_k=x_{k+n}$ is in the grid. The computational domain is $[x_0,x_{2N}]=[2^{-\f{1}{n}},2^{\f{N}{n}}]$. Thanks to the properties of the eigenvector $\cU$ established in~\cite{BCG,DG}, we have $\cU$ quickly vanishing toward $0$ and infinity, so that the truncation does not lead to an important error.

\

For the numerical scheme, it is more convenient to consider the function
\[w(t,x):=xu(t,x)e^{-\lambda t}=xv(t,x),\]
which is solution to the conservative equation
\[\f{\p}{\p t} w(t,x) +\f{\p}{\p x} (xw(t,x)) +B(x)w(t,x)=2B(2x)w(t,2x).\]
The conservation law reads
\[\int_0^\infty w(t,x)\,dx=\int_0^\infty w(0,x)\,dx\]
and we also have the contraction property
\[\|w(t,\cdot)\|_{L^1}\leq\|w(0,\cdot)\|_{L^1}.\]
We consider the semi-implicit scheme with splitting given by
\[\f{w_k^{l+\f12}-w_k^l}{\delta t}+\frac{x_kw_k^l-x_{k-1}w_{k-1}^l}{x_k-x_{k-1}}+B_kw_k^{l+\f12}=0,\quad 1\leq k\leq 2N,\]
\[\f{w_k^{l+1}-w_k^{l+\f12}}{\delta t}=2B_{k+n}w_{k+n}^{l+\f12},\quad 1\leq k\leq 2N,\]
where $B_k:=B(x_k),$ and the influx boundary condition chosen to keep the conservation property at the discrete level
\[w_0^l=\f{x_{2N}}{x_0}w_{2N}^l+\f1{x_0}\sum_{k=1}^n(x_k-x_{k-1})B_kw_k^{l+\f12}.\]
\begin{lemma}
The numerical scheme is conservative in the sense that for all $l\geq0$
\[\sum_{k=1}^{2N}(x_k-x_{k-1})w_k^{l+1}=\sum_{k=1}^{2N}(x_k-x_{k-1})w_k^{l}.\]
\end{lemma}
In order to avoid diffusivity of the numerical scheme we choose the CFL condition
\[\delta t=\frac{\delta x}{1+\delta x}.\]
Indeed under this condition the discretization of the transport term sends exactly a point of the grid on the next point of the grid (the discrete transport follows the characteristics).
Under this CFL condition, the first step of the scheme can be written as
\[w_k^{l+\f12}=\f{1}{1+\delta t B_k}\f{\delta t}{\delta x}w_{k-1}^l\]
which leads to the condensed form of the full scheme
\begin{equation}\label{eq:full1}w_k^{l+1}=\f{1}{1+\delta t B_k}\f{x_{k-1}}{x_k}w_{k-1}^l+\f{2\delta t B_{k+n}}{1+\delta t B_{k+n}}\f{\delta t}{\delta x}w_{k+n-1}^l,\quad 1\leq k\leq 2N,
\end{equation}
and
\begin{equation}\label{eq:full2}w_0^l=(1+\delta t B_1)\bigg[\f{x_{2N}}{x_0}w_{2N}^l+\sum_{k=1}^{n-1}\f{x_k}{x_0}\f{\delta t B_{k+1}}{1+\delta t B_{k+1}}w_k^l\bigg].\end{equation}
This scheme is clearly positive.
Together with the discrete conservation law we deduce that it is a contraction for the discrete $L^1$ norm $\|\cdot\|_1$ defined for a vector $u=(u_k)_{1\leq k\leq2N}$ by
\[\|u\|_1:=\sum_{k=1}^{2N}(x_k-x_{k-1})|u_k|=\frac{\delta x}{1+\delta x}\sum_{k=1}^{2N}x_k|u_k|.\]
\begin{theorem}[Convergence in the $L^1$ norm]\label{theo:numconv}
Consider that $B$ is continuous.
Let $ u^{\rm{in}}\in E_\infty$ such that $\cA  u^{\rm{in}}\in E_\infty,$ and assume that the associated solution $w(t,x)$ belongs to $C^2_b([0,+\infty)\times(0,+\infty)).$
Let $w_k^l$ be the numerical solution obtained by the iteration rule~\eqref{eq:full1}--\eqref{eq:full2} and with the initial data $u_k^0= u^{\rm{in}}(x_k).$
Then for all $r>0$ there exists a constant $C_r>0$ such that for all $T>0$
\[\sup_{t_l\leq T}\|\mathbf e^l\|_1\leq C_rT\big(2^{\f Nn-\f{\log n}{\log2}}+2^{-r\f Nn}\big),\]
where $\mathbf e$ is the ``error'' vector defined by $\mathbf e^l_k=w_k^l-w(t_l,x_k).$
\end{theorem}
This is a convergence result since if $n$ and $N/n$ tend to infinity in such a way that $\f{\log n}{\log2}-\f Nn\to+\infty,$
then the error tends to zero.
For instance if we take $N=\lfloor\f{\e}{\log2}n\log n\rfloor$ with $\e\in(0,1)$ we get a speed of convergence of order $n^{\e-1}+n^{-r\e}.$
Choosing $r=\f1\e-1$ we obtain an order $n^{\e-1}$ for any $\e\in(0,1),$ meaning that the scheme is ``almost'' of order 1 in $n.$
\begin{proof}
We write the scheme in a condensed form $w^{l+1}=Aw^l$ where $A$ is the iteration matrix.
The contraction property reads $\|A\|\leq1,$ and it implies the stability of the scheme.
Now we prove the consistency.
Taylor expansions give
\[w(t_{l+\f12},x_k)=w(t_l,x_k)+\f{\delta t}2\p_tw(t_{l+\f12},x_k)+O(\delta t^2)\]
\[x_kw(t_l,x_k)=x_{k-1}w(t_l,x_{k-1})+(x_k-x_{k-1})\p_x(xw)(t_l,x_k)+O((x_k-x_{k-1})^2)\]
and so
\begin{align*}
w(t_{l+\f12},x_k)=\f{x_{k-1}}{x_k}w(t_l,x_{k-1})+\delta t\Big[\f12\p_tw(t_{l+\f12}&,x_k)+\p_x(xw)(t_l,x_k)\Big]\\
&+O(\delta t^2+(\delta t)(x_k-x_{k-1})).
\end{align*}
We get
\begin{align*}
w(t_{l+\f12},x_k)&=\f{1}{1+\delta t B_k}\f{x_{k-1}}{x_k}w(t_l,x_{k-1})\\
&\quad+\f{\delta t}{1+\delta t B_k}\Big[\f12\p_tw(t_{l+\f12},x_k)+\p_x(xw)(t_l,x_k)+B_kw(t_{l+\f12},x_k)\Big]
\\ &\quad\quad+O(\delta t^2+(\delta t)(x_k-x_{k-1})).
\end{align*}
Now from
\[w(t_{l+1},x_k)=w(t_{l+\f12},x_k)+\f{\delta t}2\p_tw(t_{l+\f12},x_k)+O(\delta t^2)\]
we deduce
\begin{align*}
w(t_{l+1},x_k)&=\f{1}{1+\delta t B_k}\f{x_{k-1}}{x_k}w(t_l,x_{k-1})+\f{2\delta t B_{k+n}}{1+\delta t B_{k+n}}\f{\delta t}{\delta x}w(t_l,x_{k+n-1})\\
&\quad+\f{\delta t}{1+\delta t B_k}\Big[\f12\p_tw(t_{l+\f12},x_k)+\p_x(xw)(t_l,x_k)+B_kw(t_{l+\f12},x_k)\Big]\\
&\quad\quad+\f{\delta t}2\p_tw(t_{l+\f12},x_k)-\f{2\delta t B_{k+n}}{1+\delta t B_{k+n}}\f{\delta t}{\delta x}w(t_l,x_{k+n-1})
\\&\quad\quad\quad +O(\delta t^2+(\delta t)(x_k-x_{k-1})).
\end{align*}
It remains to estimate
\begin{align*}
\e_k^l&:=\f12\Big[\f{1}{1+\delta t B_k}+1\Big]\p_tw(t_{l+\f12},x_k)+\f{1}{1+\delta t B_k}\p_x(xw)(t_l,x_k)\\
&\qquad+\f{B_k}{1+\delta t B_k}w(t_{l+\f12},x_k)-\f{2 B_{k+n}}{1+\delta t B_{k+n}}\f{1}{1+\delta x}w(t_l,x_{k+n-1})
\end{align*}
and the boundary condition
\[\e_0^l:=w(t_l,x_0)-(1+\delta t B_1)\bigg[\f{x_{2N}}{x_0}w(t_l,x_{2N})+\sum_{k=1}^{n-1}\f{x_k}{x_0}\f{\delta t B_{k+1}}{1+\delta t B_{k+1}}w(t_l,x_k)\bigg].\]
Using that
\[\p_x(xw)(t_l,x_k)=\p_x(xw)(t_{l+\f12},x_k)+O(\delta t)\]
\[w(t_l,x_{k+n-1})=w(t_{l+\f12},x_{k+n})+O(\delta t+(x_k-x_{k-1}))\]
and
\[\p_tw(t_{l+\f12},x_k)+\p_x(xw)(t_{l+\f12},x_k)+B_kw(t_{l+\f12},x_k)-2B_{k+n}w(t_{l+\f12},x_{k+n})=0\]
we get
\begin{align*}
|\e_k^l|&\leq\f{1}{1+\delta t B_k}\bigg[\f{\delta t B_k}{2}|\p_tw(t_{l+\f12},x_k)|\\
&\qquad+\Big|\f{1+\delta t B_k}{1+\delta t B_{k+n}}\f{1}{1+\delta x}-1\Big|2B_{k+n}w(t_{l+\f12},x_{k+n})\bigg]
+O(\delta t+(x_k-x_{k-1}))\\
&\leq\f{1}{1+\delta t B_k}\bigg[\f{\delta t B_k}{2}\|\cA  u^{\rm{in}}\|_{E_\infty}x_k\cU(x_k)\\&\qquad+\f{\delta t B_k+\delta x+\delta t B_{k+n}+\delta t\delta x B_{k+n}}{(1+\delta t B_{k+n})(1+\delta x)}2B_{k+n}\| u^{\rm{in}}\|_{E_\infty}x_{k+n}\cU(x_{k+n})\bigg]\\
&\hspace{83mm}+O(\delta t+(x_k-x_{k-1}))\\
&\leq\delta t\bigg[\f{\|\cA  u^{\rm{in}}\|_{E_\infty}}{2}\max_k(x_kB_k\cU(x_k)) \\ 
&\qquad+2\| u^{\rm{in}}\|_{E_\infty}\max_{k}(B_{k+n}(B_k+1+2B_{k+n})x_{k+n}\cU(x_{k+n}))\bigg]\\
&\hspace{83mm}+O(\delta t+(x_k-x_{k-1}))\\
&=O(\delta t+(x_k-x_{k-1})),
\end{align*}
where we have used that $|w(t,x)|=|x T_tu^{\rm{in}}(x)|\leq\|u^{\rm{in}}\|_{E_\infty}x\,\cU(x)$ and
\[|\p_tw(t,x)|=|x\p_tT_t u^{\rm{in}}(x)|=|xT_t\cA  u^{\rm{in}}(x)|\leq\|\cA  u^{\rm{in}}\|_{E_\infty}x\,\cU(x).\]
The boundedness of $x\mapsto xB(x)\cU(x)$ and $x\mapsto 2B(2x)(1+B(x)+2B(2x))x\,\cU(2x)$ is a consequence of the assumptions~\eqref{hyp:B} on the continuous function $B$ and the estimates on $\cU$ available in~\cite{DG,BCG}.
Similarly for the boundary condition we have
\begin{align*}
|\e_0^l|&\leq\| u^{\rm{in}}\|_{E_\infty}x_0\cU(x_0)+(1+B_1)\bigg[x_{2N}^2\| u^{\rm{in}}\|_{E_\infty}x_{2N}\cU(x_{2N})\\
&\hspace{60mm}+\delta t\f{\| u^{\rm{in}}\|_{E_\infty}}{x_0}\sum_{k=1}^{n-1}x_k^2B_{k+1}\cU(x_k)\bigg]
\end{align*}
which is small when $\frac Nn$ is large since $x_0=2^{-\f Nn}, x_n=2x_0, x_{2N}=2^{\f Nn}$ and we know from~\cite{DG} that for all $r\in\R,$ when $x\to0,$
\[\cU(x)=O(x^r)\]
and from~\cite{BCG} that when $x\to+\infty,$
\[\cU(x)=O(e^{-\f{K_1}{\gamma_1} x^{\gamma_1}}).\]
More precisely for all $r\in\R$ we have
\[|\e_0^l|=O\Big(2^{-(1+r)\f Nn}+2^{3\f Nn}\exp\big(-\tfrac{K_1}{\gamma_1}2^{\gamma_1\f Nn}\big)+2^{-(1+\gamma_0+r)\f Nn}\Big)=O(2^{-r\f Nn}).\]
We conclude by the standard argument of Lax which deduces convergence from stability and consistency.
By definition we have
\[x_k-x_{k-1}=(\delta x) x_k\leq (\delta x) x_{2N}=(2^{\f 1n}-1)2^{\f nN}\sim\f{\log2}{n}2^{\f Nn}\]
\[\delta t=\f{\delta x}{1+\delta x}=1-2^{-\f1n}\sim\f{\log 2}{n}\]
so that we get
\[\|\mathbf e^{l+1}\|_1\leq \|A\mathbf e^l\|_1+\delta t\, O\Big(\f{2^{\f Nn}}{n}+2^{-r\f Nn}\Big)\leq\|\mathbf e^l\|_1+\delta t\, O\big(2^{\f Nn-\f{\log n}{\log2}}+2^{-r\f Nn}\big)\]
and we conclude by iteration, using that $\mathbf e^0=0.$
\end{proof}

\

\subsection{Illustration}

We illustrate here first the case $B(x)=x^2$: in Figure~\ref{fig:eigen}, we draw the real part of the first eigenvectors, taken for $k=0,1,2$. The oscillatory behaviour will depend on the projection of the initial condition on the space generated by $(\cU_k)$: it will be stronger if the coefficients for $k\neq 0$ are large compared to the projection on $\cU_0.$ We show two results for two different initial condition (Figure~\ref{fig:init}, Left and Right respectively), one a peak very close to the Dirac delta in $x=2$ and the other very smooth. In both cases, the solution oscillates, as showed in Figures~\ref{fig:osci:max} and~\ref{fig:osci:size}, though since the projections on $X$ are very different (with a much higher projection coefficient on the positive eigenvector for the smooth case than for the sharp case) these oscillations take very different forms. In the second case, they are so small that for any even slightly diffusive numerical scheme they are absorbed by the diffusion, leading to a seemingly convergence towards the dominant positive eigenvector. We also see that the equation is no more regularizing: discontinuities remain asymptotically for the Heaviside case.

\begin{figure}[ht]
\centering
\includegraphics[width=0.95\textwidth]{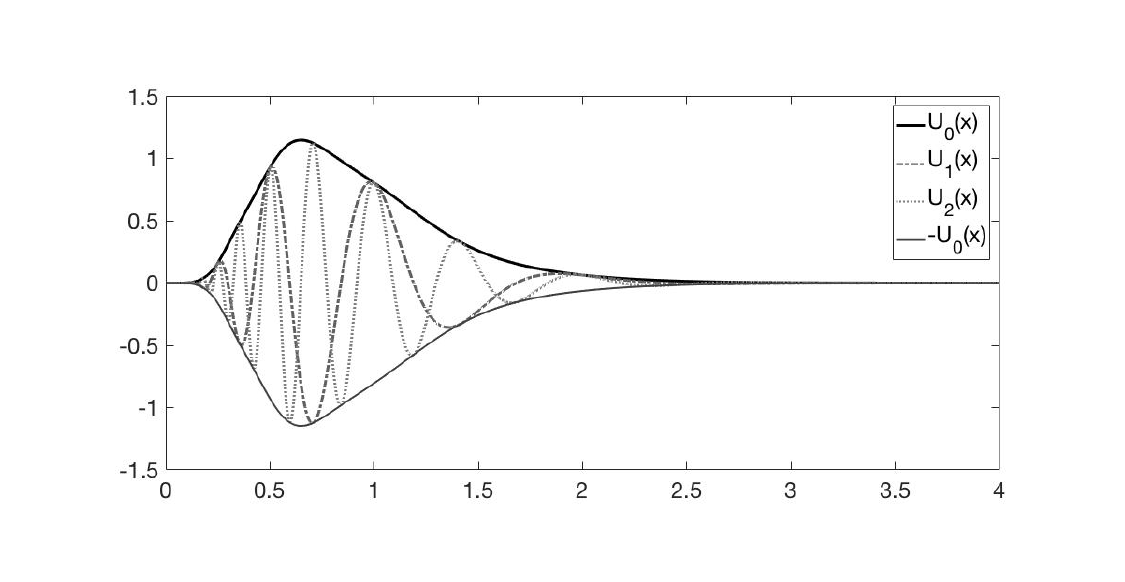}
\caption{\label{fig:eigen} The real part for the three first eigenvectors $\cU_0,\,\cU_1,\,\cU_2$ for $B(x)=x^2$. We see the oscillatory behaviour for $\cU_1$ and $\cU_2$.}
\end{figure}

\begin{figure}[ht]
\centering
\includegraphics[width=0.48\textwidth]{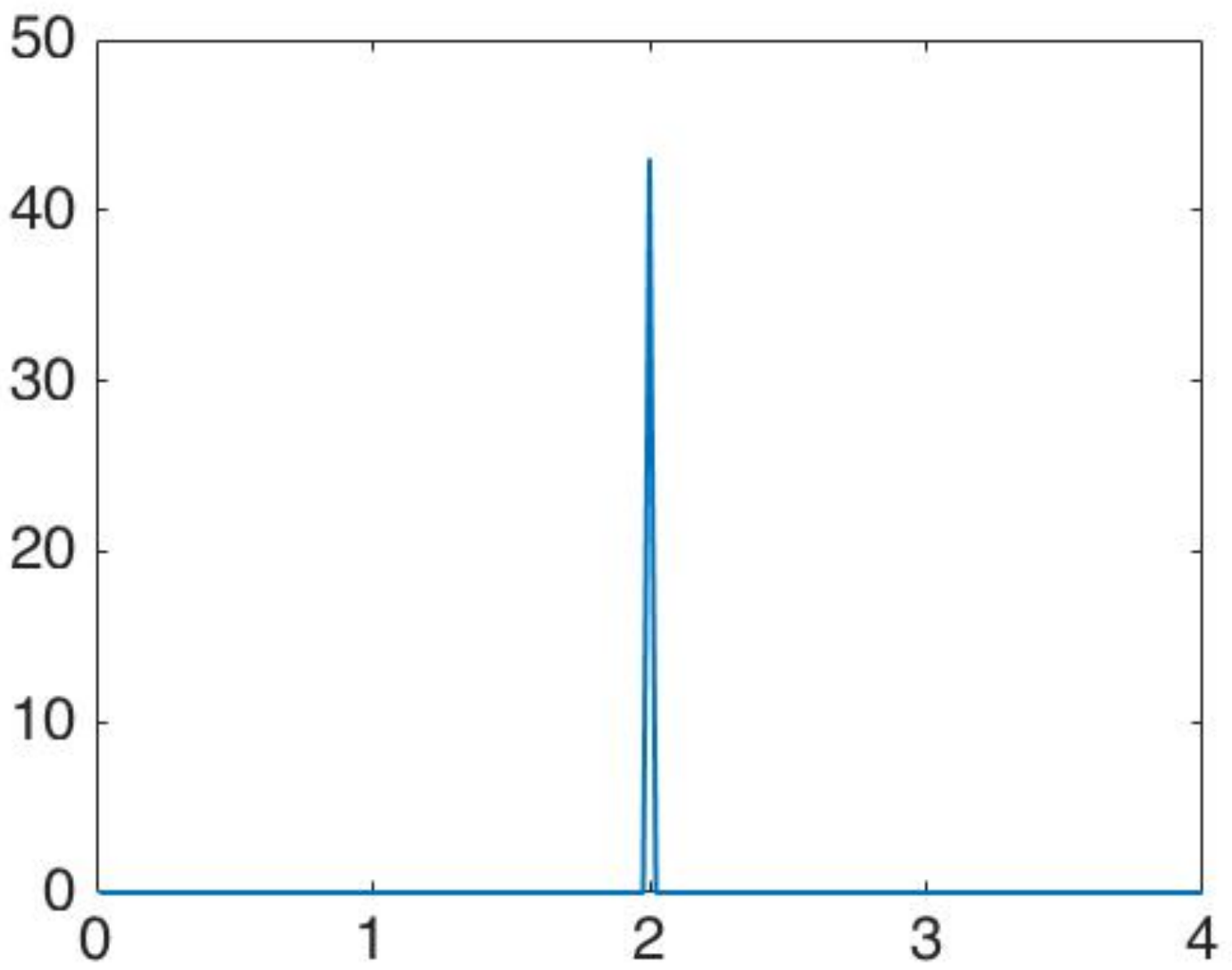}
\includegraphics[width=0.48\textwidth]{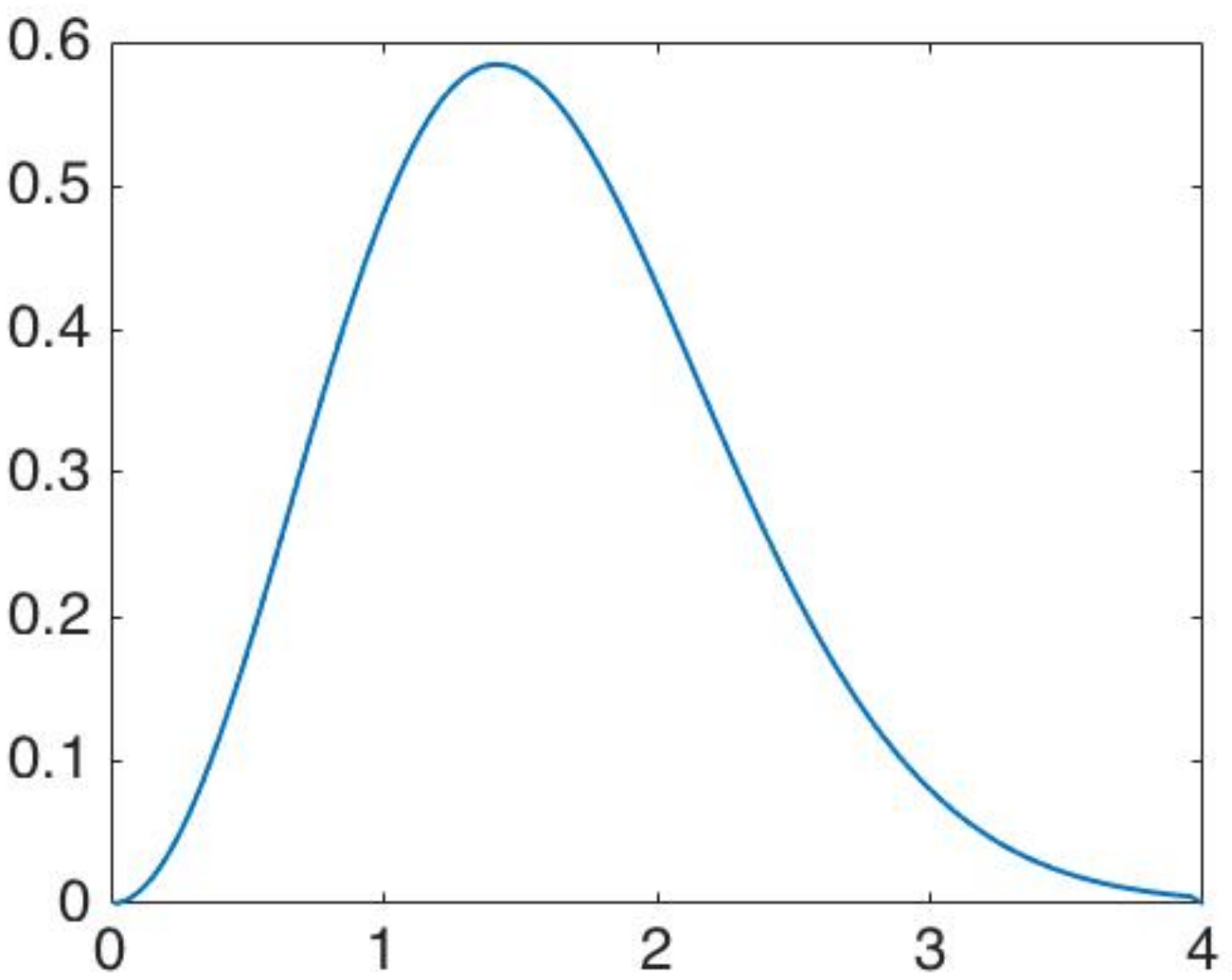}
\caption{\label{fig:init} Two different initial conditions.  \newline Left: peak in $x=2.$ Right: $ u^{\rm{in}} (x)=x^2\exp(-x^2/2)$.}
\end{figure}

\begin{figure}[ht]
\centering
\includegraphics[width=0.48\textwidth]{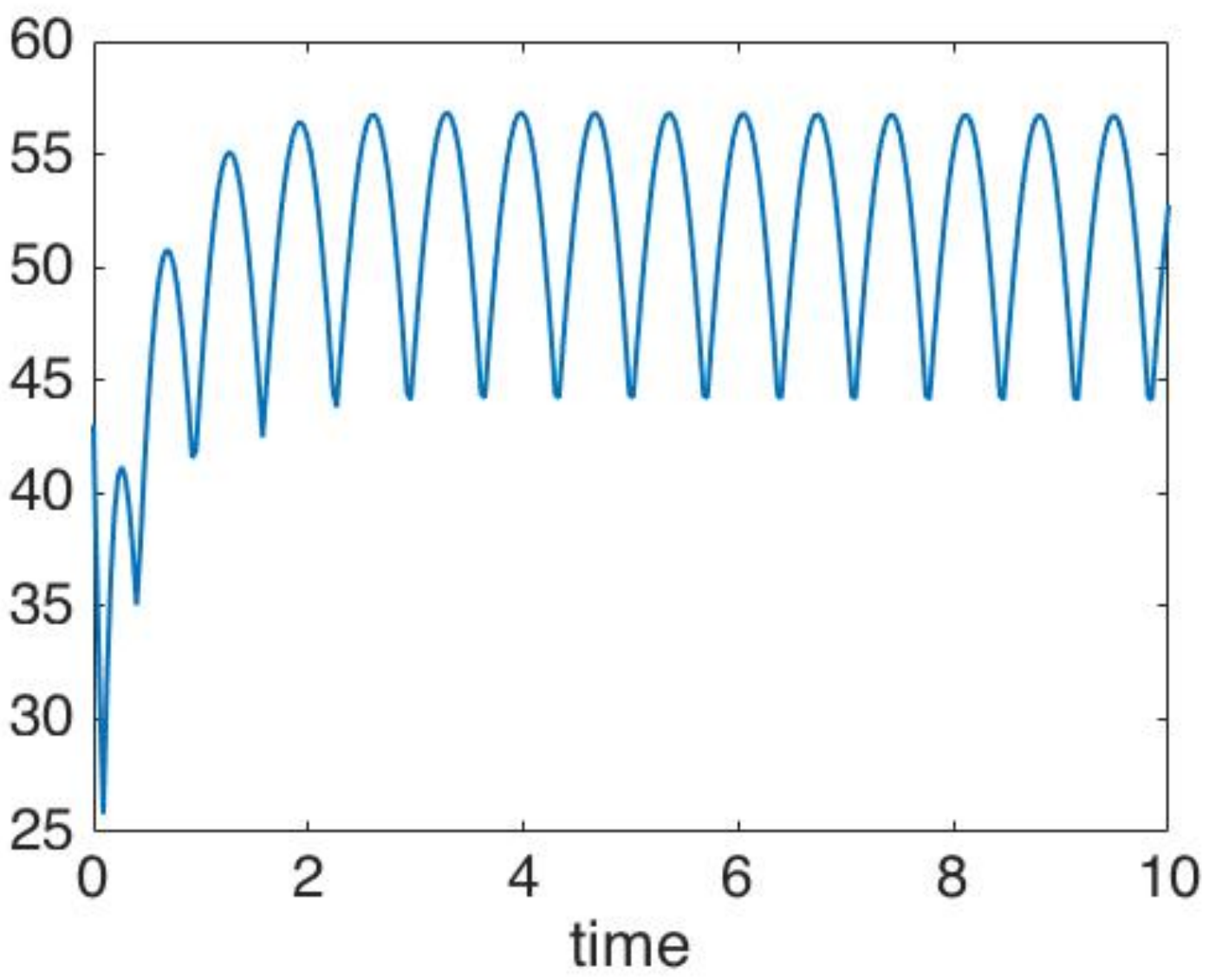}
\includegraphics[width=0.48\textwidth]{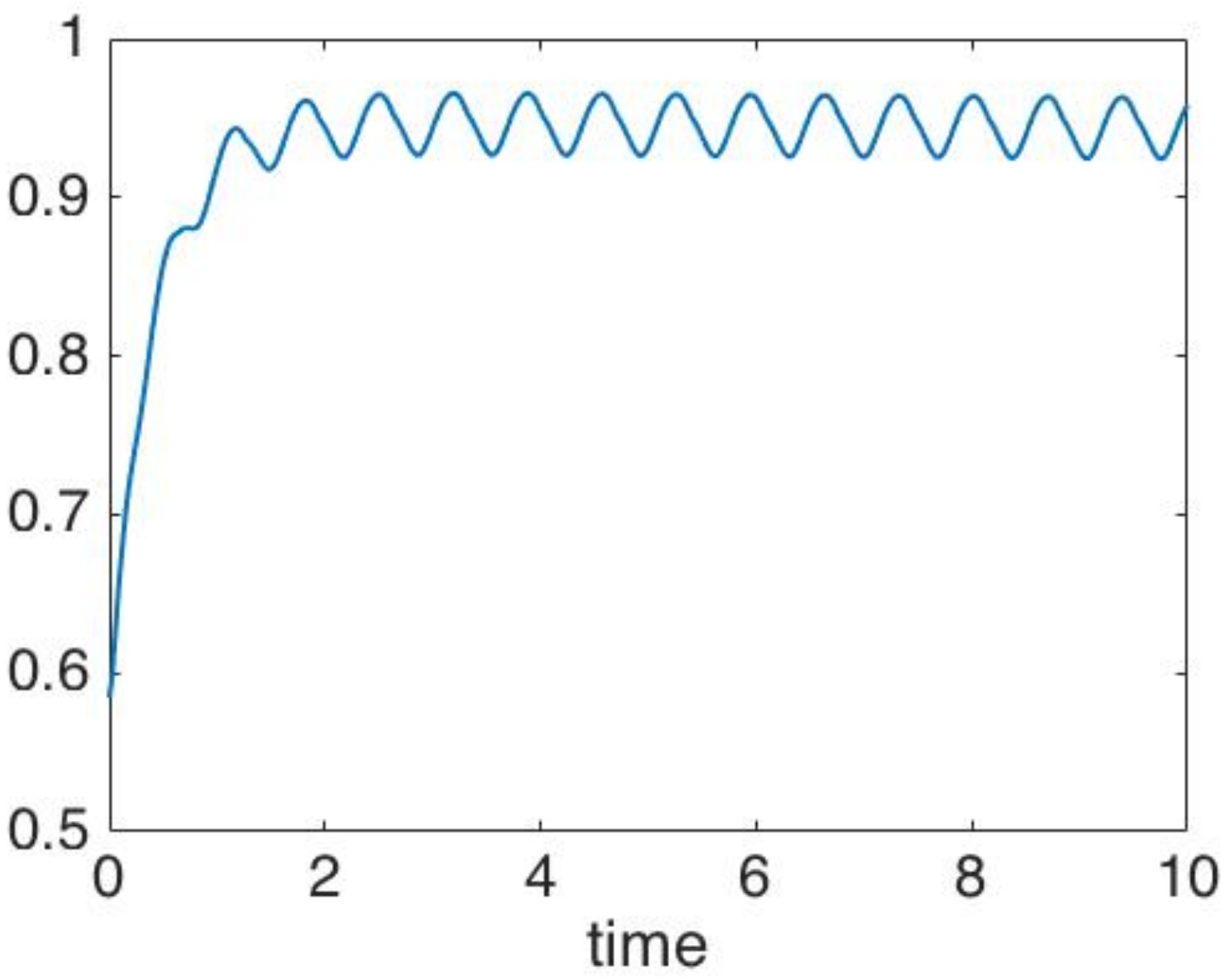}
\caption{\label{fig:osci:max} Time evolution of $\max\limits_{x>0} u(t,x)e^{-t}$. \newline Left: for the peak as initial condition. Right: for the smooth initial condition.}
\end{figure}

\begin{figure}[ht]
\centering
\includegraphics[width=0.48\textwidth]{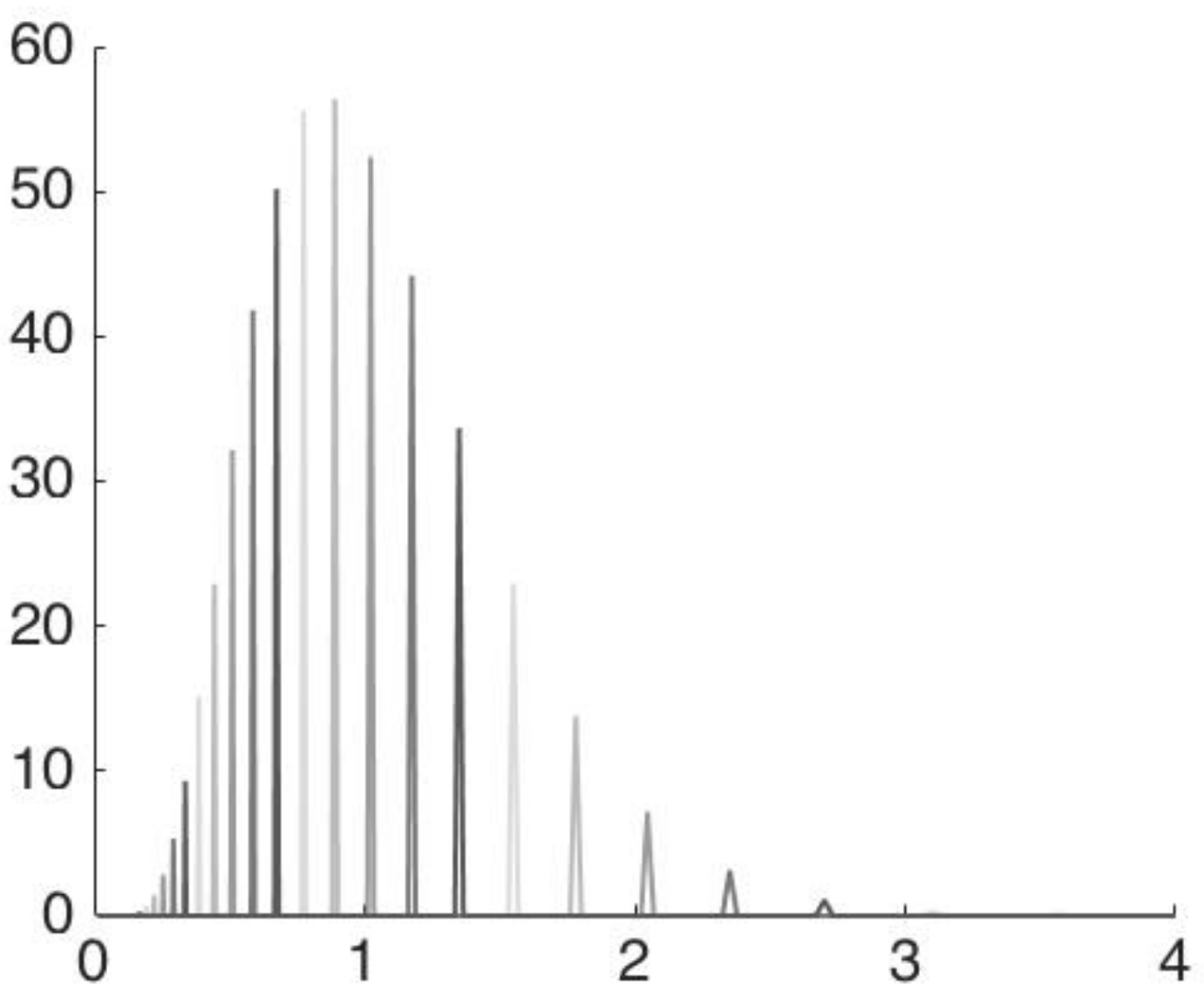}
\includegraphics[width=0.48\textwidth]{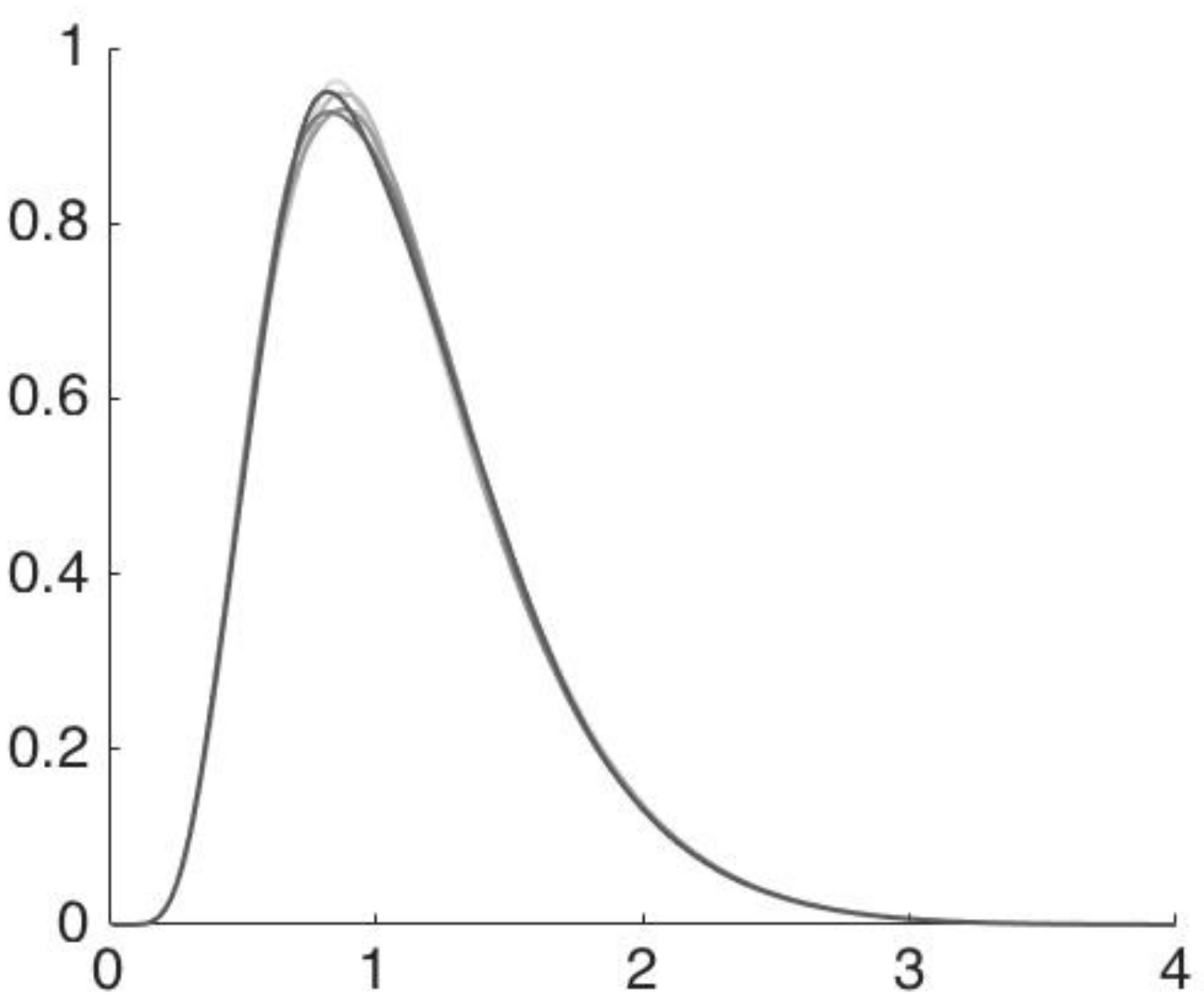}
\caption{\label{fig:osci:size} Size distribution $u(t,x)e^{-t}$ at five different times (each time is in a different grey). Left: for the peak as initial condition. Right: for the smooth initial condition.}
\end{figure}

To explore the speed of convergence in the Theorem~\ref{theo:asympto:GRE} and Corollary~\ref{cor:mean_ergodicity}, we choose $B(x)=x^3,$ take a very smooth initial condition (Figure~\ref{fig:SpectralGap} Left), with $\int\limits_0^\infty u^{\rm{in}} (x)xdx=(u^{\rm{in}},{\mathcal U}_0)=1,$ for which the coefficients $(u^{\rm{in}},{\mathcal U}_k)$ decrease rapidly with $k$, so that $R_t u^{\rm{in}}$ is very well estimated by the series truncated for $k=5.$ We then take a refined grid with $n=500$ and $N=\lfloor\f{n \log n}{2\log 2} \rfloor=2241.$ To estimate
${\mathcal U}_0,$ we use that $\f{1}{\log 2} \int\limits_t^{t+\log2} u(s,x)e^{-s} ds$ tends to ${\mathcal U}_0$, and take this limit value to define ${\mathcal U}_0^n \sim {\mathcal U}_0$ and accordingly $\cU_k^n\sim\cU_k.$ We then define the estimate for $R_t u^{\rm{in}}$ as $(R_t u^{\rm{in}})^n=\sum\limits_{k=-5}^5 (u^{\rm{in}},{\mathcal U}_k^{n}) e^{\f{2ik\pi}{\log2}t} {\mathcal U}_k^{n},$ and define the two error terms in the discrete norm  $E_2^n$ defined as $E_2$:
\begin{equation}\label{def:errornum}
\begin{array}{c}
{\text{Error}}_{E_2^n}^2:=\biggl\Vert u^{n}(t,x)e^{-t}-(R_t u^{\rm{in}})^n\biggr\Vert^2_{E_2^n},\\  {\text{Error Mean}}_{E_2^n}^2:=\biggl\Vert \f{1}{\log 2}
 \int\limits_t^{t+\log2} u^{n}(s,x)e^{-s} ds - {\mathcal U}_0^n \biggr\Vert^2_{E_2^n},
\end{array}
\end{equation}
where $u^n$ is the numerical approximation of $u.$
We observe several phases, which illustrate exactly the theory. First, a very fast decay of the quantity ${\text{Error}}_{E_2^n},$ linked to its  initial very high value since the constant $C$ such that $u^{\rm{in}}\leq C {\mathcal U}_0$ is very large. Then we have a phase of exponential decay for both ${\text{Error}}_{E_2^n}$ and ${\text{Error Mean}}_{E_2^n}$ (the linear decay in Figure~\ref{fig:SpectralGap} Right), corresponding to a spectral gap, as proved in~\cite{GreinerNagel} for the  assumption of a compact support, which is satisfied here due to the truncation. Of note, this phase lasts much more for ${\text{Error Mean}}_{E_2^n}$ than for ${\text{Error}}_{E_2^n},$ most probably due to averaging errors in the non-oscillatory solution. The final phase is either a plateau for ${\text{Error Mean}}_{E_2^n}$ (linked to our definition of ${\mathcal U}_0^n$) or a quadratic increase for ${\text{Error}}_{E_2^n}$, which is linked to the fact that the convergence constant in Theorem~\ref{theo:numconv} depends linearly on the final time.
\begin{figure}[ht]
\centering
\includegraphics[width=0.48\textwidth]{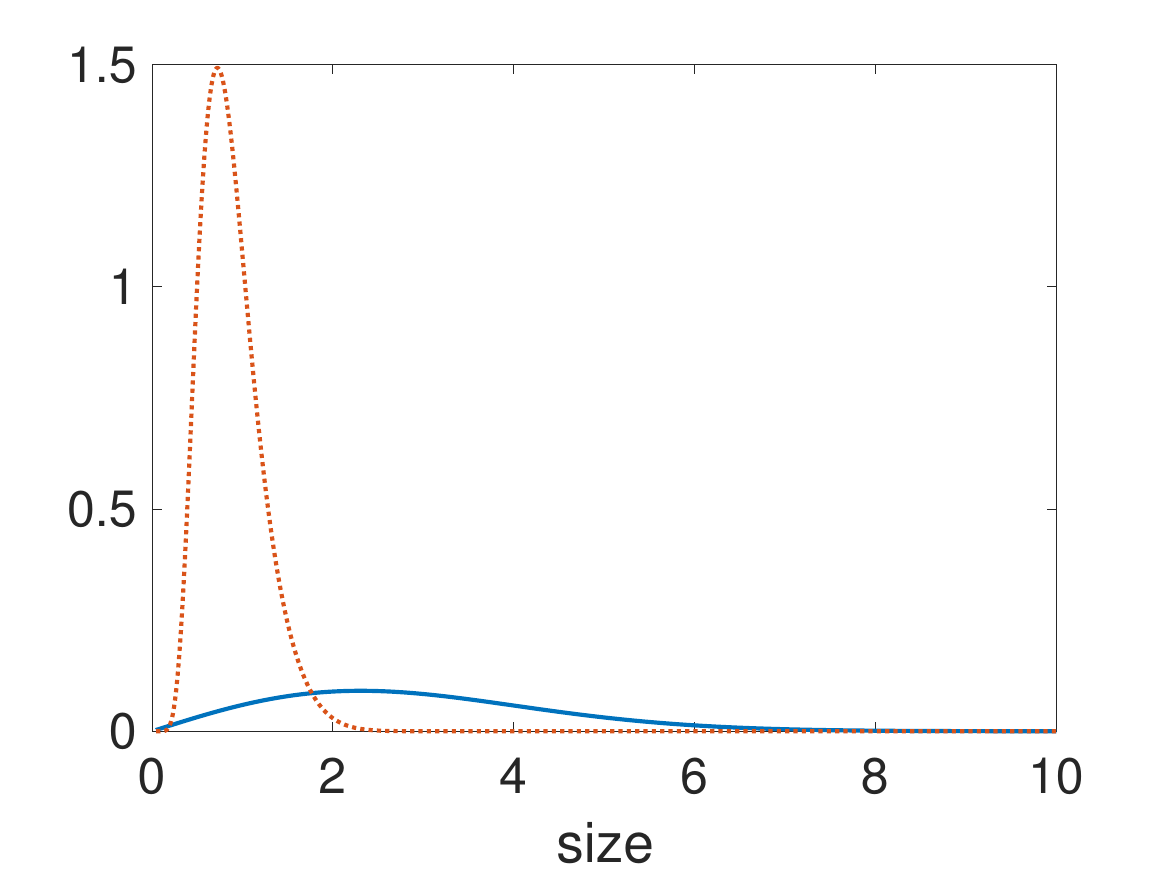}
\includegraphics[width=0.48\textwidth]{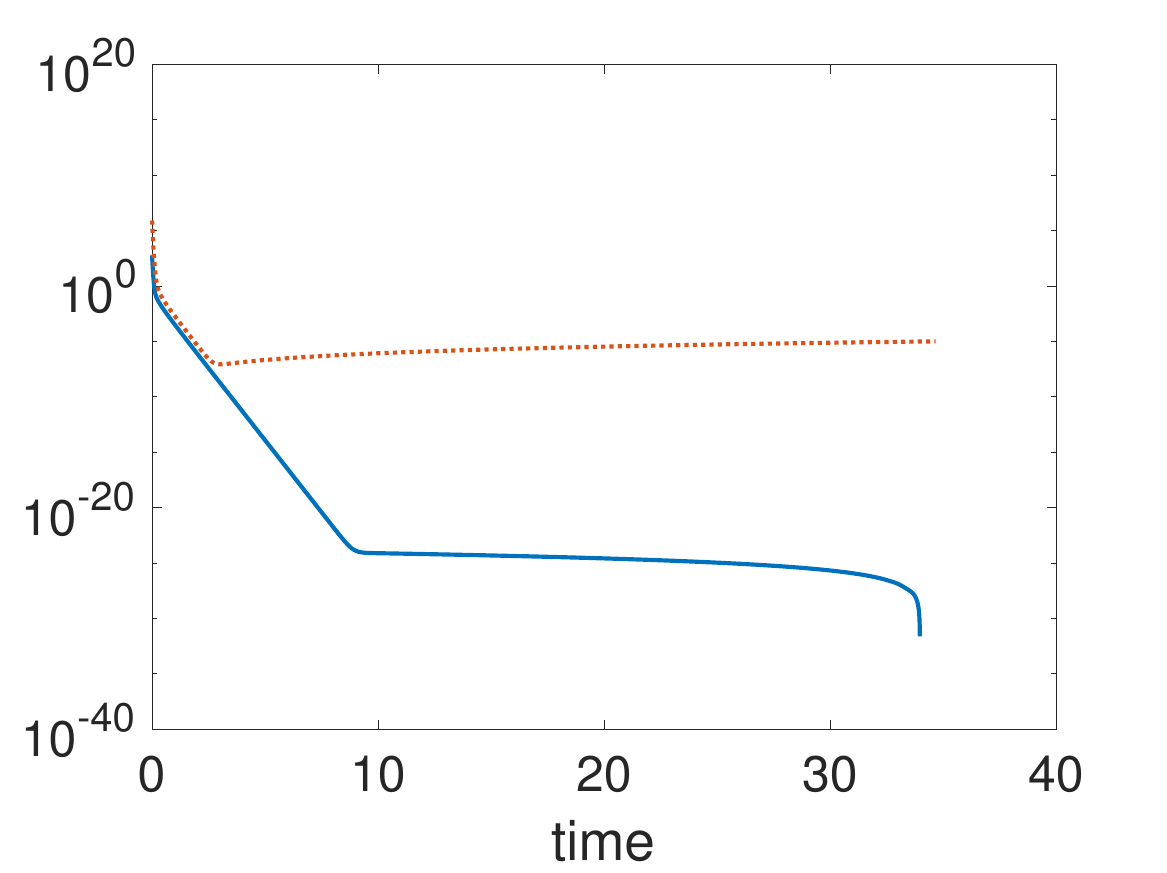}
\caption{\label{fig:SpectralGap} Left: initial distribution (full blue line) and dominant eigenvector (doted red line), for $B(x)=x^3$. We see that the constant such that $u^{\rm{in}}\leq {\mathcal U}_0$ is very large. Right: time evolution of Error$_{E_2^n}$ (doted red line) and Error Mean$_{E_2^n}$ (full blue line), in a log scale for the ordinates.}
\end{figure}

\section*{Discussion}
We studied here the asymptotic behaviour of a non-hypocoercive case of the growth-fragmentation equation. In this case,  the growth being exponential and the division giving rise to two perfectly equal-sized offspring, the descendants of a given cell all remain in a countable set of characteristics. This results in a periodic behaviour, the solution tending to its projection on the span of the dominant eigenvectors. Despite this, we were able to adapt the proofs based on general relative entropy inequalities, which provide an explicit expression for the limit.

\medskip

Our result could without effort be generalised to the conservative case, where only one of the offspring is kept at each division: in Equation~\eqref{eq:main}, the term $4 B(2x)u(2x)$ is then replaced by $2 B(2x)u(2x)$. The consequence is then simply that the dominant eigenvalue is zero, a simple calculation shows that the dominant positive eigenvector is $x\,\cU(x),$ and all the study is unchanged.

Equation~\eqref{eq:main} may also be viewed as a Kolmogorov equation of a piecewise deterministic Markov process, \emph{i.e.} as the equation satisfied by the expectation of the empirical measure of this process, see~\cite{Cloez,Haas}. Our study corresponds exactly to the case without variability in the growth rate studied in~\cite{DHKR}. In~\cite{DHKR}, a convergence result towards an invariant measure for the distribution of new-born cells is proved (this measure being $x B(x) \cU (x)$ up to a multiplicative constant). However, this does not contradict the above study, because the convergence result concerns successive generations and not a time-asymptotics. A deterministic equivalent corresponds to studying the behaviour of a time-average of the equation.
Corollary~\ref{cor:mean_ergodicity} confirms that if we rescale the solution by $e^{-t}$ and average it over a time-period, it does converge towards $\cU$.

\medskip

Our result could also easily be extended to the case where the division kernel is self similar, \emph{i.e.} $k(y,dx)=\f{1}{y}k_0(\f{dx}{y})$, and is a sum of Dirac masses specifically linked by the following relation (see Condition~H in~\cite{DE}): $Supp(k_0)=\Sigma$ where $\Sigma$ is such that
\begin{eqnarray*}
&&\hskip -0.8cm\exists L \in \N^* \cup \{+\infty\},\, \exists \theta \in (0, 1),\,\exists (p_\ell) _{ \ell\in \N , \; \ell\leq L }\subset \N,\, 0<p_\ell <p_{\ell +1}\, \forall \ell \in \N, \;\ell \leq L-1,\\
&&\hskip -0.8cm\Sigma=\left\{\sigma_\ell\in(0, 1);\sigma_\ell=\theta^{p_\ell }\right\}, \qquad (p_\ell)_{0\leq \ell \leq L}\text{ are setwise coprime.}
\end{eqnarray*}
This condition expresses the fact that all the descendants of a given individual evolve permanently on the same countable set of characteristic curves. The case of binary fission into two equal parts  corresponds to  $L=1$ and $p_L=\f{1}{2}.$ Note also that for the same reason, an oscillatory behaviour also happens for the coagulation equation in the case of the so-called \emph{diagonal kernel}~\cite{LaurencotNiethammerVelasquez}.

Other generalisations may also be envisaged, for instance to enriched equations, or to other growth rate functions satisfying $g(2x)=2g(x)$, see~\cite{DiekmannHeijmansThieme1984}, that is, functions of the form $g(x)=x \Phi (\log(x))$ where $\Phi$ a $\log (2)$ periodic function. 

\medskip

An interesting future work could consist in strengthening the convergence result in Theorem~\ref{theo:asympto:GRE}.
Indeed this result does not provide any rate of decay and the speed of convergence may depend on the initial data $ u^{\rm{in}}.$
A uniform exponential convergence would be ensured for instance by~\cite[Proposition C-IV.2.13]{nagel1},
provided that one can prove that $0$ is pole of $\cA.$
This appears to be a difficult question, which is equivalent to the uniform stability of $(T_t)_{t\geq0}$ in $X^\bot$ or to the uniform mean ergodicity of $(T_t)_{t\geq0}$ in $E_2.$

\medskip

We chose to work in weighted $L^2$ spaces because the theory may be developed both very simply and elegantly in this framework, where the terms are interpreted in terms of scalar product and Fourier decomposition. The asymptotic result of Theorem~\ref{theo:asympto:GRE} could most probably be generalised  to weighted $L^1$ spaces, by considering the entropy inequality with an adequate convex functional. Developing a theory in terms of measure-valued solutions, in the same spirit as in~\cite{GwiazdaWiedemann}, would also be of interest, since the equation has no regularising effect:  an initial Dirac mass leads to a countable set of Dirac masses at any time.

\vspace{8mm}

\noindent{\bf Acknowledgments}

M.D. has been supported by the ERC Starting Grant SKIPPER$^{AD}$ (number 306321).
P.G. has been supported by the ANR project KIBORD, ANR-13-BS01-0004, funded by the French Ministry of Research.
We thank Odo Diekmann and Rainer Nagel for their very useful suggestions.


\begin{thebibliography}{10}

\bibitem{nagel1}
W.~Arendt, A.~Grabosch, G.~Greiner, U.~Groh, H.~P. Lotz, U.~Moustakas,
  R.~Nagel, F.~Neubrander, and U.~Schlotterbeck.
\newblock {\em One-parameter semigroups of positive operators}, volume 1184 of
  {\em Lecture Notes in Mathematics}.
\newblock Springer-Verlag, Berlin, 1986.

\bibitem{Arino}
O.~Arino.
\newblock Some spectral properties for the asymptotic behavior of semigroups
  connected to population dynamics.
\newblock {\em SIAM Rev.}, 34(3):445--476, 1992.

\bibitem{BCG}
D.~Balagu\'e, J.~A. Ca\~nizo, and P.~Gabriel.
\newblock Fine asymptotics of profiles and relaxation to equilibrium for
  growth-fragmentation equations with variable drift rates.
\newblock {\em Kinet. Relat. Models}, 6(2):219--243, 2013.

\bibitem{Banasiak2}
J.~Banasiak.
\newblock On a non-uniqueness in fragmentation models.
\newblock {\em Math. Methods Appl. Sci.}, 25(7):541--556, 2002.

\bibitem{BanasiakArlotti}
J.~Banasiak and L.~Arlotti.
\newblock {\em Perturbations of positive semigroups with applications}.
\newblock Springer Monographs in Mathematics. Springer-Verlag, London, 2006.

\bibitem{BanasiakLamb}
J.~Banasiak and W.~Lamb.
\newblock The discrete fragmentation equation: Semigroups, compactness and
  asynchronous exponential growth.
\newblock {\em Kinetic Relat. Models}, 5(2):223--236, 2012.

\bibitem{BanasiakPichorRudnicki}
J.~Banasiak, K.~Pich{\'o}r, and R.~Rudnicki.
\newblock Asynchronous exponential growth of a general structured population
  model.
\newblock {\em Acta Appl. Math.}, 119(1):149--166, 2012.

\bibitem{Bell}
G.~I. Bell.
\newblock Cell growth and division: {III}. conditions for balanced exponential
  growth in a mathematical model.
\newblock {\em Biophys. J.}, 8(4):431--444, 1968.

\bibitem{BellAnderson}
G.~I. Bell and E.~C. Anderson.
\newblock Cell growth and division: I. a mathematical model with applications
  to cell volume distributions in mammalian suspension cultures.
\newblock {\em Biophys. J.}, 7(4):329--351, 1967.

\bibitem{BernardGabriel_2}
E.~Bernard and P.~Gabriel.
\newblock Asynchronous exponential growth of the growth-fragmentation equation
  with unbounded fragmentation rate.
\newblock arXiv:1809.10974.

\bibitem{Bertoin}
J.~Bertoin.
\newblock The asymptotic behavior of fragmentation processes.
\newblock {\em J. Eur. Math. Soc.}, 5(4):395--416, 2003.

\bibitem{BertoinWatson}
J.~Bertoin and A.~Watson.
\newblock Probabilistic aspects of critical growth-fragmentation equations.
\newblock {\em Adv. in Appl. Probab.}, 9 2015.

\bibitem{CCM}
M.~J. C\'aceres, J.~A. Ca\~nizo, and S.~Mischler.
\newblock Rate of convergence to an asymptotic profile for the self-similar
  fragmentation and growth-fragmentation equations.
\newblock {\em J. Math. Pures Appl.}, 96(4):334--362, 2011.

\bibitem{Cloez}
B.~Cloez.
\newblock Limit theorems for some branching measure-valued processes.
\newblock {\em Adv. in Appl. Probab.}, 49(2):549--580, 2017.

\bibitem{DiekmannHeijmansThieme1984}
O.~Diekmann, H.~Heijmans, and H.~Thieme.
\newblock On the stability of the cell size distribution.
\newblock {\em J. Math. Biol.}, 19:227--248, 1984.

\bibitem{DE}
M.~Doumic and M.~Escobedo.
\newblock Time asymptotics for a critical case in fragmentation and
  growth-fragmentation equations.
\newblock {\em Kinetic Relat. Models}, 9(2):251--297, 2016.

\bibitem{DG}
M.~Doumic and P.~Gabriel.
\newblock Eigenelements of a general aggregation-fragmentation model.
\newblock {\em Math. Models Methods Appl. Sci.}, 20(05):757, 2009.

\bibitem{DHKR}
M.~Doumic, M.~Hoffmann, N.~Krell, and L.~Robert.
\newblock Statistical estimation of a growth-fragmentation model observed on a
  genealogical tree.
\newblock {\em Bernoulli}, 21(3):1760--1799, 2015.

\bibitem{nagel2}
K.-J. Engel and R.~Nagel.
\newblock {\em One-parameter Semigroups for Linear Evolution Equations}.
\newblock Springer-Verlag, New York, 2000.

\bibitem{EscoMischler3}
M.~Escobedo, S.~Mischler, and M.~Rodriguez~Ricard.
\newblock On self-similarity and stationary problem for fragmentation and
  coagulation models.
\newblock {\em Ann. Inst. H. Poincar\'e Anal. Non Lin\'eaire}, 22(1):99--125,
  2005.

\bibitem{GS14}
P.~Gabriel and F.~Salvarani.
\newblock Exponential relaxation to self-similarity for the superquadratic
  fragmentation equation.
\newblock {\em Appl. Math. Lett.}, 27:74--78, 2014.

\bibitem{GreinerNagel}
G.~Greiner and R.~Nagel.
\newblock Growth of cell populations via one-parameter semigroups of positive
  operators.
\newblock In {\em Mathematics applied to science}, pages 79--105. Academic
  Press, Boston, MA, 1988.

\bibitem{GwiazdaWiedemann}
P.~Gwiazda and E.~Wiedemann.
\newblock Generalized entropy method for the renewal equation with measure
  data.
\newblock {\em Commun. Math. Sci.}, 15(2):577--586, 2017.

\bibitem{Haas}
B.~Haas.
\newblock Asymptotic behavior of solutions of the fragmentation equation with
  shattering: an approach via self-similar {M}arkov processes.
\newblock {\em Ann. Appl. Probab.}, 20(2):382--429, 2010.

\bibitem{HallWake_1990}
A.~J. Hall and G.~C. Wake.
\newblock Functional-differential equations determining steady size
  distributions for populations of cells growing exponentially.
\newblock {\em J. Austral. Math. Soc. Ser. B}, 31(4):434--453, 1990.

\bibitem{Heijmans1985}
H.~Heijmans.
\newblock An eigenvalue problem related to cell growth.
\newblock {\em J. Math. Anal. Appl.}, 111:253--280, 1985.

\bibitem{LaurencotNiethammerVelasquez}
P.~{Lauren{\c c}ot}, B.~{Niethammer}, and J.~J.~L. {Vel{\'a}zquez}.
\newblock {Oscillatory dynamics in Smoluchowski's coagulation equation with
  diagonal kernel}.
\newblock {\em Kinetic Relat. Models}, 11(4):933--952, 2018.

\bibitem{LP}
P.~{L}auren{\c{c}}ot and B.~{P}erthame.
\newblock {E}xponential decay for the growth-fragmentation/cell-division
  equation.
\newblock {\em Comm. Math. Sc.}, 7(2):503--510, 2009.

\bibitem{MMP1}
P.~Michel, S.~Mischler, and B.~Perthame.
\newblock General entropy equations for structured population models and
  scattering.
\newblock {\em C. R. Math. Acad. Sci. Paris}, 338(9):697--702, 2004.

\bibitem{MMP2}
P.~Michel, S.~Mischler, and B.~Perthame.
\newblock General relative entropy inequality: an illustration on growth
  models.
\newblock {\em J. Math. Pures Appl. (9)}, 84(9):1235--1260, 2005.

\bibitem{mischler:frag}
S.~Mischler and J.~Scher.
\newblock Spectral analysis of semigroups and growth-fragmentation equations.
\newblock {\em Ann. Inst. H. Poincar\'e Anal. Non Lin\'eaire}, 33(3):849--898,
  2016.

\bibitem{PPS3}
K.~Pakdaman, B.~Perthame, and D.~Salort.
\newblock Adaptation and fatigue model for neuron networks and large time
  asymptotics in a nonlinear fragmentation equation.
\newblock {\em J. Math. Neurosci.}, 4(14):1--26, 2014.

\bibitem{BP}
B.~Perthame.
\newblock {\em Transport equations in biology}.
\newblock Frontiers in Mathematics. Birkh\"auser Verlag, Basel, 2007.

\bibitem{PR}
B.~Perthame and L.~Ryzhik.
\newblock Exponential decay for the fragmentation or cell-division equation.
\newblock {\em J. Differential Equations}, 210(1):155--177, 2005.

\bibitem{SinkoStreifer2}
J.~Sinko and W.~Streifer.
\newblock A model for populations reproducing by fission.
\newblock {\em Ecology}, 52(2):330--335, 1971.

\bibitem{Villani}
C.~Villani.
\newblock Hypocoercivity.
\newblock {\em Mem. Amer. Math. Soc.}, 202(950):iv+141, 2009.

\bibitem{Zaidi_2015}
A.~A. Zaidi, B.~van Brunt, and G.~C. Wake.
\newblock A model for asymmetrical cell division.
\newblock {\em Math. Biosc. Eng.}, 12(3):491--501, 2015.

\bibitem{Zaidi_2015-b}
A.~A. Zaidi, B.~Van~Brunt, and G.~C. Wake.
\newblock Solutions to an advanced functional partial differential equation of
  the pantograph type.
\newblock {\em Proc. A.}, 471(2179):20140947, 15, 2015.

\end{thebibliography}
\end{document}